\documentclass[]{amsart}
\usepackage{amssymb,mathdots}
\usepackage{mathrsfs}
\usepackage[utf8]{inputenc}
\usepackage{amsmath}
\usepackage{enumerate}
\usepackage{url}

\def\scaledpicture#1by#2(#3scaled#4){{
\dimen0=#1  \dimen1=#2
\divide\dimen0 by 1000 \multiply\dimen0 by #4
\divide\dimen1 by 1000 \multiply\dimen1 by #4
\picture \dimen0 by \dimen1 (#3 scaled #4)}}
\def\dfigure#1by#2(#3scaled#4offset#5:#6)
    {\medskip
     \vglue 2mm minus 2mm
     $$
       \hbox{
         \hglue#5
         {\scaledpicture #1 by #2 (#3 scaled #4)}
       }
     $$
     \par\goodbreak
     \vglue 2mm minus 2mm
     \medskip}

\def\qmod#1#2{{\hbox{}^{\displaystyle{#1}}}\!\big/\!\hbox{}_{
\displaystyle{#2}}}

\def\resto#1#2{{
#1\hskip 0.4ex\vline_{\hskip 0.2ex\raisebox{-0,2ex}
{{${\scriptstyle #2}$}}}}}

\def\C{{\mathbb C}}

\def\E{{\mathbb E}}
\def\F{{\mathbb F}}

\def\L{{\mathbb L}}

\def\N{{\mathbb N}}

\def\P{{\mathbb P}}

\def\R{{\mathbb R}}
\def\Z{{\mathbb Z}}

\def\union{\mathop{\bigcup}}
\def\qed {\hfill\vrule height6pt width6pt depth0pt \smallskip}

\def\map{\longrightarrow}
\def\textmap#1{\mathop{\vbox{\ialign{
                                  ##\crcr
      ${\scriptstyle\hfil\;\;#1\;\;\hfil}$\crcr
      \noalign{\kern 1pt\nointerlineskip}
      \rightarrowfill\crcr}}\;}}
\def\bigtextmap#1{\mathop{\vbox{\ialign{
                                  ##\crcr
      ${\hfil\;\;#1\;\;\hfil}$\crcr
      \noalign{\kern 1pt\nointerlineskip}
      \rightarrowfill\crcr}}\;}}
      
\newcommand{\cal}{\mathcal}
\def\textlmap#1{\mathop{\vbox{\ialign{
                                  ##\crcr
      ${\scriptstyle\hfil\;\;#1\;\;\hfil}$\crcr
      \noalign{\kern-1pt\nointerlineskip}
      \leftarrowfill\crcr}}\;}}

\def\cg{{\mathfrak c}}

\def\fg{{\mathfrak f}}
\def\g{{\mathfrak g}}

\def\ig{{\mathfrak i}}
\def\jg{{\mathfrak j}}

\def\wg{{\mathfrak w}}

\def\Fg{{\mathfrak F}}

\newtheorem{sz}{Satz}[section]
\newtheorem{thry}[sz]{Theorem}
\newtheorem{pr}[sz]{Proposition}
\newtheorem{re}[sz]{Remark}
\newtheorem{co}[sz]{Corollary}
\newtheorem{dt}[sz]{Definition}
\newtheorem{lm}[sz]{Lemma}

\begin{document}
 
\def\tr{\mathrm {Tr}}
\def\End{\mathrm {End}}
\def\Aut{\mathrm {Aut}}
\def\Spin{\mathrm {Spin}}
\def\U{\mathrm{U}}
\def\SU{\mathrm {SU}}
\def\SO{\mathrm {SO}}
\def\PU{\mathrm {PU}}
\def\GL{\mathrm {GL}}
\def\spin{\mathrm {spin}}
\def\u{\mathrm {u}}
\def\su{\mathrm {su}}
\def\so{\mathrm {so}}
\def\pu{\mathrm {pu}}
\def\Pic{\mathrm {Pic}}
\def\Iso{\mathrm {Iso}}
\def\NS{\mathrm{NS}}
\def\deg{\mathrm {deg}}
\def\Hom{\mathrm{Hom}}
\def\Herm{\mathrm{Herm}}
\def\Vol{{\rm Vol}}
\def\pf{{\bf Proof: }}
\def\id{ \mathrm{id}}
\def\Im{\mathrm{Im}}
\def\im{\mathrm{im}}
\def\rk{\mathrm {rk}}
\def\ad{\mathrm {ad}}
\def\spc{\mathrm{Spin}^c}
\def\niq{=\kern-.18cm /\kern.08cm}
\def\Ad{\mathrm {Ad}}
\def\RSU{\R\mathrm{SU}}
\def\ad{{\rm ad}}
\def\dva{\bar\partial_A}
\def\da{\partial_A}
\def\p{{\rm p}}
\def\sp{\Sigma^{+}}
\def\sm{\Sigma^{-}}
\def\spm{\Sigma^{\pm}}
\def\smp{\Sigma^{\mp}}
\def\oo{{\scriptstyle{\cal O}}}
\def\ooo{{\scriptscriptstyle{\cal O}}}
\def\sw{Seiberg-Witten }
\def\pa{\partial_A\bar\partial_A}
\def\Dr{{\raisebox{0.17ex}{$\not$}}{\hskip -1pt {D}}}
\def\gr{{\scriptscriptstyle|}\hskip -4pt{\g}}
\def\subsetint{{\  {\subset}\hskip -2.45mm{\raisebox{.28ex}
{$\scriptscriptstyle\subset$}}\ }}

\title[real Theta line bundles of Klein surfaces]
{Abelian Yang-Mills theory on Real tori and Theta divisors of Klein surfaces}
\author{Christian Okonek \& Andrei Teleman}
\address{Christian Okonek:  Institut f\"ur Mathematik, Universit\"at Z\"urich,
Winterthurerstrasse 190, CH-8057 Z\"urich, e-mail: okonek@math.uzh.ch}

\address{Andrei Teleman: 
CMI,   Aix-Marseille Universit\'e,  LATP, 39  Rue F. Joliot-Curie, F-13453
Marseille Cedex 13,   e-mail: teleman@cmi.univ-mrs.fr
}

 \begin{abstract} The purpose of this paper is to compute  determinant index bundles of certain families of Real Dirac type operators on Klein surfaces as elements in the corresponding Grothendieck  group of Real line bundles in the sense of Atiyah. On a Klein surface these determinant index bundles have a natural holomorphic description  as theta line bundles. In particular we compute the first Stiefel-Whitney classes of the corresponding fixed point bundles on the real part 
 of the Picard torus. The
computation of these classes is important, because they control to a large
extent the orientability of certain moduli spaces in Real gauge theory and Real algebraic geometry.
 
 \end{abstract} 
\thanks{The second author has been partially supported by the ANR project MNGNK, decision 
Nr.  ANR-10-BLAN-0118}
\maketitle

\tableofcontents

\setcounter{section}{-1}
\section{Introduction}

This paper is the first in a series in which we will develop a mathematical gauge theory in low dimensions in the presence of a Real structure in the sense of Atiyah \cite{A1}. In dimension 2 this will lead to a theory of gauged linear $\sigma$-models defined on Klein surfaces with values in symplectic quotients endowed with    Real structures. In dimension 4 this will yield  a Real version of Seiberg-Witten theory, which should have applications to the classification of Real algebraic surfaces.

It is well-known that one of the main issues  in connection  with Real structures is the orientability problem. In gauge theory orientability of moduli spaces is controlled by a numerical index and a determinant index bundle associated with a family of Dirac type operators. These determinant index bundles come with a natural Real structure in the sense of Atiyah, and the main problem is to determine their equivalence classes as elements in the Grothendieck group which classifies these Real bundles. In the present paper we will concentrate on the 2-dimensional case. In this situation the relevant determinant index bundles can be identified with certain natural theta line bundles of Klein surfaces.

Recall that a Klein surface is a pair $(C,\iota)$ consisting of a closed Riemann surface $C$ and an anti-holomorphic involution $\iota:C\to C$. The topological type of a Klein surface is determined by the triple $(g,r,a)$, where $g$ is the genus of $C$, $r$ the number of connected components  of the fixed point locus $C^\iota$, and $a$ is the orientation obstruction of the $\iota$-quotient, i.e., $a(C,\iota)=0$ when $C/\langle\iota\rangle$ is orientable  and   $a(C,\iota)=1$ when not. The Real structure $\iota$ induces a Real structure $\hat\iota:\Pic(C)\to\Pic(C)$  on the Picard group, given by $\hat\iota([{\cal L}]):=[\iota^*\bar{\cal L}]$.  The geometric theta divisor
$$\Theta:=\{[{\cal L}]\in\Pic^{g-1}(C)|\ h^0({\cal L})>0\}
$$
is $\hat\iota$-invariant and therefore defines a natural Real holomorphic line bundle ${\cal L}:={\cal O}_{\Pic^{g-1}(C)}(\Theta)$ on $\Pic^{g-1}(C)$.
There are two important families of Real holomorphic line bundles which one gets by translating $\Theta$ to $\Pic^0(C)$: One can either choose a Real theta characteristic $[\kappa]\in\Pic^{g-1}(C)$ and put 
$${\cal L}_{\kappa}:={\cal O}_{\Pic^0(C)}(\Theta-[\kappa])\ ,$$
 or, when $C^\iota\ne\emptyset$, one can choose a point $p_0\in C^\iota$ and define 
$${\cal L}_{p_0}:={\cal O}_{\Pic^0(C)} (\Theta-[{\cal O}_C((g-1)p_0)])\ .$$
 Both of these types  of theta line bundles are determinant index bundles of  families of perturbed Dirac operators associated with Spin- or Spin$^c$-structures. The underlying smooth Real line bundles $(L_\kappa,\tilde\iota_{{\cal L}_\kappa})$ and $(L_{p_0},\tilde \iota_{{\cal L}_{p_0}})$  define elements in the Grothendieck cohomology group \cite{G} $H^1_{\Z_2}(\Pic^0(C),\underline{S}^1(1))\simeq H^2_{\Z_2}(\Pic^0(C),\underline{\Z}(1))$ classifying Real line bundles on the Real torus $(\Pic^0(C),\hat\iota)$.

The cohomology group $H^1_{\Z_2}(\Pic^0(C),\underline{S}^1(1))$   comes with two natural morphisms:
$$c:H^1_{\Z_2}(\Pic^0(C),\underline{S}^1(1))\to H^2(\Pic^0(C),\Z)\ , $$ $$  w: H^1_{\Z_2}(\Pic^0(C),\underline{S}^1(1))\to H^1(\Pic^0(C)^{\hat \iota},\Z_2)\ ,
$$
defined by $c([L,\tilde\iota]):=c_1(L)$, $w([L,\tilde\iota]):= w_1(L^{\tilde\iota})$.

Whereas the first Chern classes $c_1({\cal L}_\kappa)$ and $c_1({\cal L}_{p_0})$  are well known, and can be calculated by the Atiyah-Singer index theorem for families, there is no analogous index theorem which would compute  the first Stiefel-Whitney classes of the corresponding fixed point bundles.  However, it is precisely these first Stiefel-Whitney classes $w([L_\kappa,\tilde\iota_{{\cal L}_\kappa}])$ and $w([L_{p_0},\tilde \iota_{{\cal L}_{p_0}}])$ which control the orientability of the corresponding moduli spaces.

Our strategy for computing $w([L_\kappa,\tilde\iota_{{\cal L}_\kappa}])$ and $w([L_{p_0},\tilde \iota_{{\cal L}_{p_0}}])$ is to first determine the Appell-Humbert data of ${\cal L}_\kappa$, then extract $w_1({\cal L}_\kappa^{\tilde\iota_\kappa})$ from these data and, in a third  step,  compare $(L_\kappa,\tilde\iota_{{\cal L}_\kappa})$ with $(L_{p_0},\tilde \iota_{{\cal L}_{p_0}})$. The final result is a completely explicit formula for $w_1({\cal L}_\kappa^{\tilde\iota_{{\cal L}_\kappa}})$ in terms of $w_1(\kappa^{\tilde\iota_\kappa})$, and for $w_1({\cal L}_{p_0}^{\tilde \iota_{{\cal L}_{p_0}}})$ in terms of the component of $C^\iota$ in which $p_0$ lies.

We illustrate the effectiveness  of our results by studying the simplest possible moduli spaces associated with $C$, i.e., the symmetric powers $S^d(C)$ of the curve itself. The induced Real structure on the moduli space $S^d(C)$ is the obvious one; its fixed locus $S^d(C)^{\iota}$ of $\iota$-invariant points decomposes as a disjoint union of several connected components, some of  which are orientable,  whereas others are not. The orientability of the different components is controlled by the first Stiefel-Whitney class of certain Real theta bundles, which can be computed using our explicit formulas alluded to above.  
\vspace{2mm}

Let us now briefly describe the content of the four sections of the article. 

In Section 1 we construct, using gauge theoretical techniques, two families of Dolbeault operators on a Riemann surface $C$, and we show that for a	 Klein surface $(C,\iota)$ with $C^\iota\ne\emptyset$ the corresponding determinant line bundles have natural $\hat\iota$-Real structures. The obtained $\hat\iota$-Real bundles can be identified with the underlying differentiable line bundles  of ${\cal L}_{p_0}$ and  ${\cal L}_{\kappa}$. At the end of the section we describe an important example of a Real gauged linear sigma model. The corresponding moduli spaces can be regarded as Uhlenbeck  type compactifications  of moduli spaces of sections in projective bundles over Klein surfaces. We also explain how the orientability of the real  part of the corresponding moduli spaces (which are Quot schemes defined over $\R$) is controlled by the first Stiefel-Whitney classes of certain  real  determinant line bundles on $\Pic^0(C)^{\hat\iota}$. The main problem here is to compute the relevant determinant line bundles as   Real holomorphic line bundles on $\Pic^0(C)$, in terms of the geometric theta divisor.

In section 2 we apply Grothendieck's formalism \cite{G} of equivariant sheaf cohomology to identify the set of isomorphism classes of  Real line bundles on a Real topological space  $(X,\iota)$ with the cohomology group   $H^1_{\Z_2}(X,\underline{S}^1(1))$. We obtain a fundamental short exact sequence for this cohomology group, which will allow us to compute it explicitly in several important cases, namely for a Klein surface and a Real torus. The section continues with an interesting general result  which   identifies the   isomorphism classes  of  Real line bundles on a compact Real Riemannian manifold $(X,\iota)$  with the connected components of the fixed point locus of the  induced involution on the  moduli spaces of Yang-Mills connections on $X$.  This  result has a  complex geometric version in which the Yang-Mills moduli space is replaced by the Picard group. 

The third section begins with with an important result  which describes the moduli space of Yang-Mills connections on a torus in terms of  generalized Appell-Humbert data. Generalized Appell-Humbert data are linear algebra  data which specify a unique representative in each gauge equivalence class of Yang-Mills connections. These linear algebra  data have a differential geometric interpretation  as the curvature and the holonomy along standard loops of the corresponding Yang-Mills connection. This is relevant for our purposes, because  the Stiefel-Whitney class of a real line bundle can be identified with the holonomy representation of an $\mathrm{O}(1)$-connection on it. 
The main result of the third section is a classification theorem for  Real line bundles $(L,\tilde\iota)$ on a Real torus   $(T,\iota)$, i.e., the explicit computation of $H^1_{\Z_2}(T,\underline{S}^1(1))$ in terms of characteristic classes. To the best of our knowledge this is the first explicit computation of such a Grothendieck group in a non-trivial case. A first step in the proof  is a   universal {\it difference formula} which describes the jump of the Stiefel-Whitney class of the restrictions of the fixed point line bundle $ L^{\tilde\iota}$ when one passes from one connected component  of  $T^\iota$  to another.

Section 4 concerns Real theta line bundles on complex tori. First we  specialize   our description of the abelian Yang-Mills moduli space to the case of a complex torus.  Using the Kobayashi-Hitchin correspondence and our explicit description we obtain a new proof of the classical Appell-Humbert theorem describing the Picard group of a complex torus. For our purposes the most important point is that we get a clear geometric understanding of the complex analytic data intervening in the canonical factor of automorphy which appear in the Appell-Humbert theorem. This allows us to compute the holonomy of abelian Hermite-Einstein connections along all standard loops. 
  The section continues with a subsection  
in which we compute the Appell-Humbert data which determine the relevant theta line bundles of Klein surfaces.  First we consider the Spin case, i.e., the symmetric theta line bundles ${\cal L}_\kappa$, where $\kappa$ is Real theta characteristic.  Using Riemann's singularity theorem we show that the Appell-Humbert data describing ${\cal L}_\kappa$ are given in terms of the intersection form of the curve and Mumford's theta form $q_\kappa$. Combining this with our previous results we obtain already an explicit formula for the Stiefel-Whitney class $w([{  L}_\kappa,\tilde\iota_{{\cal L}_\kappa}])$. The problem with this formula is that it involves  Mumford's theta form, which is an algebraic geometric object. Using results of Atiyah \cite{A2}, Johnson \cite{J}, Libgober \cite{L}, and a new geometric construction we  obtain a purely topological interpretation of the theta form $q_\kappa$ in terms of $w_1(\kappa^{\tilde\iota_\kappa})$\footnote{It has been brought to our attention by the referee, that this result, our Theorem \ref{main}, -- in the case of {\it effective} Real theta characteristics --  has been proved before by S. Natanzon. His proof -- using real Fuchsian groups and their liftings -- is nicely explained in his book \cite{N}. For the convenience of the reader we will summarize some of his results in our terminology after Corollary \ref{comessatti}.} In a final step we compare $w([{  L}_\kappa,\tilde\iota_{{\cal L}_\kappa}])$ and $w([L_{p_0},\tilde \iota_{{\cal L}_{p_0}}])$ to  obtain a corresponding formula in the $\spc$ case. These results determine explicitly the elements defined by ${\cal L}_\kappa$ and ${\cal L}_{p_0}$ in the Grothendieck group $H^1_{\Z_2}(\Pic^0(C),\underline{S}^1(1))$.

We  have included  an appendix in which we prove  a general $\Z_2$-localization formula which relates the Stiefel-Whitney numbers of a Real vector  bundle $(E,\tilde\iota)$ on a compact Real manifold  $(X,\iota)$ to  the Stiefel-Whitney classes of the real bundle $E^{\tilde \iota}$ over $X^\iota$ and the normal bundle of $X^\iota$ in $X$.

\section{Families of Dirac operators of Klein  surfaces }\label{famDirac}

\subsection{Families of  Dirac operators on a Riemann surface}

 \subsubsection{ Families of $\spc$-Dirac operators} Let $C$ be a compact Riemann surface of genus $g$. The spinor bundles of the {\it canonical} $\spc$-structure $\tau_{\rm can}$ on $C$ are $\Sigma^+_{\rm can}=\Lambda^{0}$, $\Sigma^-_{\rm can}=\Lambda^{0,1}$, and the canonical Dirac operator of $C$ associated with the canonical $\spc$-structure  is (up to the factor $\sqrt{2}$) just the Dolbeault operator
$$\bar\partial: A^{0} \map A^{0,1} \ .
$$
Using the canonical $H^2(C,\Z)$-torsor structure on the set of equivalence classes of $\spc$-structures on $C$ one obtains for any Hermitian line bundle $L$  on $C$ an $L$-twisted $\spc$-structure $\tau_L$ on $C$ whose spinor bundles are $\Sigma^+_L:=\Lambda^0(L)$, $\Sigma^-_L:=\Lambda^{0,1}(L)$. For the construction of a Dirac operator associated with $\tau_L$ one needs a  semi-connection  $\delta$ on $L$, and then the corresponding $\spc$-Dirac operator will be
$$\delta:A^0(L)\to A^{0,1}(L)\ .
$$
Varying $\delta$ in the space $\mathcal{A}^{0,1}(L)$ of semi-connections on $L$ one gets a tautological  family of elliptic operators parameterized by  $\mathcal{A}^{0,1}(L)$.

Let ${\cal G}^\C:=\mathcal{C}^\infty(C,\C^*)$ be the complex gauge group, which acts on $\mathcal{A}^{0,1}(L)$ from the right by $(\delta,g)\mapsto g^{-1}\circ\delta\circ g=\delta+g^{-1}\bar\partial g$. We are interested in descending this tautological family to the moduli space
$$\Pic^d(C)\simeq \qmod{{\cal A}^{0,1}(L)}{{\cal G}^\C}
$$
of holomorphic structures on $L$ (here $d:=\deg(L)$). The problem here is that the  group ${\cal G}^\C$ does not act freely on the  affine space ${\cal A}^{0,1}(L)$, so the trivial line bundles
$${\cal A}^{0,1}(L)\times A^0(L)\ ,\ {\cal A}^{0,1}(L)\times A^{0,1}(L)
$$
do not descend to $\Pic^d(C)$ in a natural way. In order to descend these bundles we choose a base point $p_0\in C$ and consider the reduced complex gauge group 
$${\cal G}_{p_0}^\C:=\{g\in{\cal G}^\C|\ g(p_0)=1\}\ .$$
Since this group acts freely on ${\cal A}^{0,1}(L)$ we get Fréchet vector bundles 
$$\mathscr{E}^0_{p_0}:={\cal A}^{0,1}(L)\times_{{\cal G}_{p_0}^\C} A^0(L)\ ,\ \mathscr{E}^1_{p_0}:={\cal A}^{0,1}(L)\times_{{\cal G}_{p_0}^\C} A^{0,1}(L)$$ 
over $\qmod{{\cal A}^{0,1}(L)}{{\cal G}^\C}$, and a universal family of Dolbeault operators
$$\delta_{p_0}^L:\mathscr{E}^0_{p_0}\to \mathscr{E}^1_{p_0}
$$
of index $d+1-g$ parameterized by $\Pic^d(C)$.

The determinant line bundle $\det\mathrm{ind}\ \delta_{p_0}^L$ has been extensively studied in the literature \cite{Q},   \cite{BGS1}, \cite{BGS2}, \cite{BGS3}. This line bundle has a natural holomorphic structure which can be described as follows. Consider the Poincaré line bundle 
$$\L_{p_0}:=\qmod{{\cal A}^{0,1}(L)\times L}{{\cal G}^\C_{p_0}}
$$
over $\Pic^d(C)\times C$. Then one has a canonical isomorphism
$$\det\mathrm{ind}\ \delta_{p_0}^L\simeq \det(R^0q_*(\L_{p_0}))\otimes\left[\det(R^1q_*(\L_{p_0})\right)]^\vee\ ,$$
where $q$ stands for the canonical projection $\Pic^d(C)\times C\to  \Pic^d(C)$ (see \cite{KM}, \cite{BGS1}, \cite{BGS2}, \cite{BGS3}), hence $\det\mathrm{ind}\ \delta_{p_0}^L$ has a natural holomorphic structure. Choosing a different base point $p_1$    yields a new Poincaré line bundle $\L_{p_1}$, and the two Poincaré line bundles are related by a formula of the form
$$\L_{p_1}=\L_{p_0}\otimes q^*({\cal M})\ ,
$$
where ${\cal M}$ is topologically trivial holomorphic line bundle on $\Pic^d(C)$. Using the projection formula for the functors  $R^iq_*$ one obtains  
$$\det\mathrm{ind}\ \delta_{p_1}^L\simeq \det\mathrm{ind}\ \delta_{p_0}^L\otimes {\cal M}^{\otimes(d+1-g)}\ .
$$
This shows in particular that $\det\mathrm{ind}\ \delta_{p_0}^L$ is independent of  the choice of $p_0$ when $d=g-1$; in this case $\det\mathrm{ind}\ \delta_{p_0}^L$ 
  has a canonical section whose zero locus is the geometric theta divisor  $\Theta\subset \Pic^{g-1}(C)$, hence one has
\begin{equation}\label{det-theta}\det\mathrm{ind}\ \delta_{p_0}^L\simeq{\cal O}_{\Pic^{g-1}(C)}(\Theta)\ ,
\end{equation}
independently of $p_0$.

It is easy to compare to determinant line bundles $\det\mathrm{ind}\ \delta_{p_0}^L$, $\det\mathrm{ind}\ \delta_{p_0}^{L'}$ associated with differentiable line bundles $L$, $L'$ of degrees $d$, $d'$. Let $P_0$ the underlying differentiable line bundle of ${\cal O}_C(p_0)$, and choose $L:=L'\otimes P_0^k$, where $k=d-d'$. For $L'$  we have a  Poincaré line bundle
$$\L'_{p_0}:=\qmod{{\cal A}^{0,1}(L')\times L'}{{\cal G}^\C_{p_0}} \ .
$$
We denote by $\delta_{p_0}$ the semi-connection on $P_0$ defining the holomorphic structure  of  ${\cal O}_C(p_0)$, and we define the isomorphism
$$\varphi_{p_0}:\Pic^{d'}(C)\to \Pic^{d}(C)
$$
by $\varphi_{p_0}([\delta']):=[\delta'\otimes \delta_{p_0}^{\otimes k}]$. One can easily check that
$$[\varphi_{p_0}\times\id_C]^*(\L_{p_0})\simeq \L'_{p_0}\otimes p^*({\cal O}_C(kp_0))\ ,
$$
where $p:\Pic^d(C)\times C\to C$ denotes the canonical projection. Note that 
$$p^*({\cal O}_C(kp_0))\simeq {\cal O}_{\Pic^d(C)\times C}(k\Sigma_{p_0})\ ,
$$
where $\Sigma_{p_0}:=\Pic^d(C)\times\{p_0\}$.  When $k\geq 0$  we tensor the short exact sequence
$$0\map {\cal O}_{\Pic^d(C)\times C}\map {\cal O}_{\Pic^d(C)\times C}(k\Sigma_{p_0})\map {\cal O}_{k\Sigma_{p_0}}(k\Sigma_{p_0})\map 0
$$
with $\L'_{p_0}$ and we write the corresponding long   exact sequence:
$$0\to R^0q_*(\L'_{p_0})\to R^0q_*([\varphi_{p_0}\times\id_C]^*\L_{p_0})\to\Pic^d(C)\times {\cal O}_{kp_0}\to
$$
$$\to R^1q_*(\L'_{p_0})\to R^1q_*([\varphi_{p_0}\times\id_C]^*\L_{p_0})\to 0\ .
$$
\\

For a holomorphic line bundle $\mathcal{M}$ of degree $m$ on $C$ and a fixed $s\in \Z$, we denote by $\tau_\mathcal{M}$ the {\it translation} by  $[\mathcal{M}]$   i.e., the isomorphism $\Pic^s(C)\to\Pic^{s+m}(C)$ defined by tensorizing with $\mathcal{M}$. It $\mathcal{D}$ is a divisor on $\Pic^{s+d}(C)$ we put $\mathcal{D}+[\mathcal{M}^\vee]:=\tau_\mathcal{M}^{-1}(\mathcal{D})$. Then one has 
$$\tau_\mathcal{M}^*(\mathcal{O}_{\Pic^{s+d}(C)}(\mathcal{D}))=\mathcal{O}_{\Pic^s(C)}(\mathcal{D}+[\mathcal{M}^\vee])\ ,
$$
and the isomorphism $\varphi_{p_0}:\Pic^{d'}(C)\to \Pic^{d}(C)$ introduced above is a translation $\tau_\mathcal{M}$ defined by a suitable line bundle $\mathcal{M}$.

\begin{lm} \label{ImpRed} Let $L$ be a differentiable line bundle of degree $d$ on $C$, and $p_0$ a base point. One has
$$\det\mathrm{ind}\ \delta_{p_0}^{L}\simeq {\cal O}_{\Pic^d(C)}(\Theta+[{\cal O}_C((d-g+1)p_0)])\ .$$
\end{lm}
 \begin{proof}
Using the functoriality of the functor $\det$ and the exact sequence above we get
$$ \varphi_{p_0}^*(\det\mathrm{ind}\ \delta_{p_0}^{L})\simeq \det\mathrm{ind}\ \delta_{p_0}^{L'}\ . 
$$
In  the special case $d'=g-1$ this yields
$$
\varphi_{p_0}^*(\det\mathrm{ind}\ \delta_{p_0}^{L})\simeq {\cal O}_{\Pic^{g-1}(C)}(\Theta)\ .
$$
If we denote by $\Theta+[{\cal O}_C((d-g+1)p_0)]$ the translate of the geometric theta divisor $\Theta\subset \Pic^{g-1}(C)$ by the point $[{\cal O}_C((d-g+1)p_0)]\in \Pic(C)$, the last isomorphism can be rewritten as $\det\mathrm{ind}\ \delta_{p_0}^{L}\simeq {\cal O}_{\Pic^d(C)}(\Theta+[{\cal O}_C((d-g+1)p_0)])$.
\end{proof}

In order to understand the line bundle $\det\mathrm{ind}\ \delta_{p_0}^{L}$ explicitly, we shall identify $\Pic^d(C)$ with the torus $\Pic^0(C)$ (which has an explicit description as the quotient $H^1(C,{\cal O})/H^1(C,\Z)$) using the isomorphism $\otimes {\cal O}_C(dp_0):\Pic^0(C)\to\Pic^d(C)$. We get
\begin{lm} \label{ImpRe} Via the identification $\tau_{ {\cal O}_C(dp_0)}:\Pic^0(C)\to\Pic^d(C)$ the line bundle $\det\mathrm{ind}\ \delta_{p_0}^{L}$ on $\Pic^d(C)$ corresponds to the line bundle ${\cal O}_{\Pic^0(C)}(\Theta-[{\cal O}_C((g-1)p_0)])$ on $\Pic^0(C)$.
\end{lm}

In these constructions we considered families of $\spc$-Dirac operators obtained by coupling the canonical $\spc$-Dirac operator with holomorphic line bundles.  
\subsubsection{Families of Dirac operators associated with theta characteristics} A different point of view begins with the Dirac operator associated with a fixed $\Spin$-structure on $C$. A Riemann surface of genus $g$ has $2^{2g}$ equivalence classes of $\Spin$-structures; these classes correspond bijectively to isomorphism classes of theta characteristics, i.e.,  to square roots $\kappa$ of the canonical line bundle ${\cal K}_C$ \cite{A2}.

The spinor bundles corresponding to a theta characteristic $\kappa$ are 
$$S^+=\kappa\ ,\ S^-=\Lambda^{0,1}(\kappa)\simeq \kappa^\vee\ ,$$
 and the corresponding Dirac operator is (up to the factor $\sqrt{2}$) just the Dolbeault operator $\bar\partial_\kappa:A^0(\kappa)\to A^{0,1}(\kappa)$.
A second natural way to construct a family of Dirac operators is to consider  perturbations of this $\Spin$-Dirac operator by flat line bundles. 

Associated with any   form $\eta\in A^{0,1}$ we have  a perturbed   Dirac operator 
$$ \bar\partial_\kappa+\eta: A^0(\kappa)\to A^{0,1}(\kappa)\ .$$
Factorizing again by the reduced gauge group ${\cal G}^{\C}_{p_0}$, we obtain  Fréchet bundles
$$\mathscr{F}^0_{p_0}:= {A^{0,1}}\times_{{\cal G}^\C_{p_0}} A^0(\kappa)\ ,\ \mathscr{F}^1_{p_0}:= {A^{0,1}}\times_{{\cal G}^\C_{p_0}} A^{0,1}(\kappa)
$$
over the quotient 
$$\qmod{A^{0,1}}{{\cal G}^\C_{p_0}}\simeq \qmod{A^{0,1}}{\{g^{-1}\bar\partial g|\ g\in{\cal G}^\C\}}\simeq \Pic^0(C)\ ,$$
and  a family of operators 
$$\bar\partial_{\kappa,p_0}:\mathscr{F}^0_{p_0}\to \mathscr{F}^1_{p_0}\ 
$$
parameterized by  $\Pic^0(C)$.
Note that the 
family $\bar\partial_{\kappa,p_0}$ is just the pull-back of the family $\delta^L_{p_0}:\mathscr{E}^0_{p_0}\to \mathscr{E}^1_{p_0}$ via the isomorphism $[{\cal L}_0]\to [{\cal L}_0\otimes \kappa]$, where $L$ is the underlying differentiable line bundle of $\kappa$.
This proves
\begin{lm} \label{LemmaK} Let $\kappa$ be a theta characteristic on $C$, $p_0$ a base point. One has 
$$\det\mathrm{index}\ \bar\partial_{\kappa,p_0}\simeq{\cal O}_{\Pic^0(C)}(\Theta-[\kappa])\ .$$
\end{lm}
\subsection{Real determinant line bundles and theta bundles of Klein surfaces}
\label{realdet}
Let $(C,\iota)$ be a Klein surface with $C^\iota\ne\emptyset$. The Real structure $\iota$ induces a natural Real structure $\hat\iota$ on $\Pic(C)$ mapping $[{\cal L}]$ to $[\iota^*(\bar{\cal L})]$, which preserves each component 
$\Pic^d(C)$.  

The involution $\hat\iota$ can be  constructed with gauge theoretical methods in the following way. Fix a $\iota$-Real \cite{A1} line bundle $(L,\tilde\iota)$ of degree $d$ (one can take for instance the underlying differentiable line bundle of the line bundle associated with a $\iota$-invariant divisor of degree $d$).  By definition this means that $\tilde\iota$ is a differentiable fiberwise antilinear  involution of $L$ lifting $\iota$. This involution induces anti-linear involutions $\tilde\iota^*$ on the spaces of $L$-valued forms $A^{0,q}(L)$ acting by
$$\tilde\iota^*(\sigma)(x):=\tilde\iota (\sigma(\iota(x))) \ ,\ \tilde\iota^*(\alpha\otimes\sigma):=\overline{\iota^*}(\alpha)\otimes\tilde \iota^*(\sigma)\ \ \forall\alpha\in A^{0,q}(L)\ \forall \sigma\in\Gamma(L)\ .
$$
For a semi-connection $\delta\in{\cal A}^{0,1}(L)$ put
$\tilde\iota^*(\delta):=\tilde \iota^*\circ \delta\circ \tilde\iota^*$. Using the identity
$$\tilde\iota^*(\delta\cdot g)=\tilde\iota^*(\delta)\cdot\overline{\iota^*}(g)$$
we see that the map $\delta\mapsto \tilde\iota^*(\delta)$ induces an involution ${\cal A}^{0,1}(L)/{\cal G}^\C\to {\cal A}^{0,1}(L)/{\cal G}^\C$. Via the identification ${\cal A}^{0,1}(L)/{\cal G}^\C=\Pic^d(C)$ this involution coincides with $\hat \iota$, so it is independent of the choice of the $\iota$-Real structure $\tilde\iota$ on $L$. Taking $p_0\in C^\iota$ we see that the map $g\mapsto \overline{\iota^*}(g)$ leaves the subgroup ${\cal G}_{p_0}^\C\subset {\cal G}^\C$ invariant, and the product map 
$$\tilde \iota^*\times \tilde\iota:{\cal A}^{0,1}(L)\times L\longrightarrow{\cal A}^{0,1}(L)\times L$$ 
induces an anti-holomorphic $(\hat\iota\times\iota)$-Real structure on the Poincaré line bundle $\L_{p_0}$ over $\Pic^d(C)\times C$. Regarding $\hat\iota\times\iota$ as a biholomorphism $\overline{\Pic}^d(C)\times \overline{C}\to \Pic^d(C)\times C$ and using the functoriality  of   $\det(R^0(\cdot))\otimes\left[\det(R^1(\cdot)\right)]^\vee$ with respect to biholomorphic base-change, we obtain

\begin{re} \label{RealStr1} Choosing $p_0\in C^\iota$ we get a $\hat\iota$-Real structure on the determinant line bundle  $\det\mathrm{index}\ \delta^L_{p_0}$ which is anti-holomorphic with respect to its natural holomorphic structure.
\end{re} 

This Real structure can be explicitly described  fiberwise using the fiber identifications $\det\mathrm{index}\ \delta^L_{p_0}([\delta])=\wedge^{\max}H^0({\cal L}_\delta)\otimes \wedge^{\max}H^1({\cal L}_\delta)^\vee$: it is induced by the  anti-linear isomorphisms  
$$ H^0({\cal L}_\delta)\to H^0({\cal L}_{\tilde\iota^*(\delta)})\ ,\  H^1({\cal L}_\delta) \to   H^1({\cal L}_{\tilde\iota^*(\delta)}) 
$$
given by  the operators $\tilde\iota^*$ on the spaces   $A^0(L)$ and $A^{0,1}(L)$. Here we denoted by ${\cal L}_\delta$ the holomorphic line bundle defined by the semi-connection $\delta\in {\cal A}^{0,1}(L)$.

Note that  the geometric theta divisor $\Theta$ is  invariant under $\hat\iota$, so that ${\cal O}_{\Pic^{g-1}(C)}(\Theta)$ has an obvious anti-holomorphic $\hat\iota$-Real structure. The same holds for the line bundle $ {\cal O}_{\Pic^0(C)}(\Theta-[{\cal O}_C((g-1)p_0)])$ on $\Pic^0(C)$ when $p_0\in C^\iota$. It is easy to see that, in general, two anti-holomorphic Real  structures on the same holomorphic line bundle are congruent modulo $S^1$. Together with Remark \ref{RealStr1} it follows that the  Real structure constructed on  $\det\mathrm{index}\ \delta^L_{p_0}$ with the help of a Real structure $\tilde\iota$ on $L$,  is independent of $\tilde\iota$ up to a multiplicative constant in $S^1$.  Using the isomorphism (\ref{det-theta}) and Lemma \ref{ImpRe}  we obtain
\begin{re} \label{RealIso}   For $p_0\in C^\iota$ we have  isomorphisms of $\hat\iota$-Real line bundles 
$$\left\{ \tau_{ {\cal O}_C((d+1-g)p_0)}\right\}^*(\det\mathrm{ind}\ \delta_{p_0}^L)\simeq{\cal O}_{\Pic^{g-1}(C)}(\Theta)\ ,$$
$$ \left\{ \tau_{ {\cal O}_C(dp_0)}\right\}^*(\det\mathrm{ind}\ \delta_{p_0}^L)\simeq {\cal O}_{\Pic^0(C)}(\Theta-[{\cal O}_C((g-1)p_0)]).
$$
\end{re}

Fix now a $\iota$-Real  theta-characteristic $(\kappa,\tilde\iota_\kappa)$, i.e., a square root  $\kappa$ of ${\cal K}_C$ endowed with an anti-holomorphic $\iota$-Real structure  $\tilde\iota_\kappa$.  We will see (see Proposition \ref{FC}) that the set of isomorphism  classes of such pairs $(\kappa,\tilde\iota_\kappa)$ corresponds bijectively to the finite subset of $\Pic^{g-1}(C)$ consisting of $\hat\iota$-invariant square roots of $[{\cal K}_C]$.

\begin{re} \label{RealStr2} Choosing $p_0\in C^\iota$, and a $\iota$-Real  theta-characteristic $(\kappa,\tilde\iota_\kappa)$  we get a $\hat\iota$-Real structure on the determinant line bundle 
$\det\mathrm{index}\ \bar\partial_{\kappa,p_0}$ and an isomorphism 
$$\det\mathrm{index}\ \bar\partial_{\kappa,p_0}\simeq {\cal O}_{\Pic^0(C)}(\Theta-[\kappa])
$$
 of $\hat\iota$-Real line bundles on $\Pic^0(C)$. 
\end{re}

As we explained in the introduction, our first goal is to identify the underlying  differentiable line bundles of the  $\hat\iota$-Real determinant line bundles 
$$\left\{ \tau_{ {\cal O}_C(dp_0)}\right\}^*(\det\mathrm{ind}\ \delta_{p_0}^L)\simeq   {\cal O}_{\Pic^0(C)}(\Theta-[{\cal O}_C((g-1)p_0)])$$
$$ \det\mathrm{index}\ \bar\partial_{\kappa,p_0}\simeq  {\cal O}_{\Pic^0(C)}(\Theta-[\kappa])$$  on  $(\Pic^0(C),\hat\iota)$ as elements in the  cohomology group $H^1_{\Z_2}(\Pic^0(C),\underline{S}^1(1))$, and in particular to compute the Stiefel-Whitney class  of the associated fixed point   bundles over  $\Pic^0(C)^{\hat\iota}$. Note that $\Pic^0(C)^{\hat\iota}$ is a disjoint  union of a finite family of real sub-tori of $\Pic^0(C)$ parameterized by the quotient $H^1(C,\Z)^{\iota^*}/(\id+\iota^*)H^1(C,\Z)$ (see section \ref{examples}). \vspace{3mm}

\subsection{Real gauged linear sigma models}\label{RGLSM}

The goal of this section is to describe briefly an example of the Real version of the theory of gauge theoretical Gromov-Witten invariants \cite{OT2}, which can be considered as generalizations of Witten's gauged linear sigma models.

We start by fixing    topological data $(\Sigma,M,E_0)$, where $\Sigma$ is a closed, connected, oriented 2-manifold, $M$ a Hermitian line bundle on $\Sigma$, and $E_0$ a Hermitian vector bundle of rank $r_0$ on $\Sigma$. The corresponding configuration space is 
$$\mathcal{A}(M,E_0):=\mathcal{A}(M)\times A^0(\Hom(M,E_0))\ ,$$
 where $\mathcal{A}(M)$ denotes the space of Hermitian connections on $M$.  The gauge group  $\mathcal{G}:=\mathcal{C}^\infty(\Sigma,S^1)$  acts  on $\mathcal{A}(M,E_0)$ by $(A,\varphi)g=(g^*(A),\varphi\circ g)$.  We denote by $\mathcal{B}(M,E_0)$ the quotient $\mathcal{A}(M,E_0)/\mathcal{G}$. Now we choose continuous parameters $(g,A_0)$, where $g$ is Riemannian metric on $\Sigma$, and $A_0$ a Hermitian connection on $E_0$.  The $(0,1)$-component $\bar\partial_{A_0}$  of $d_{A_0}$  defines a holomorphic structure $\mathcal{E}_0$ on $E_0$. Note  that $g$ together with the fixed orientation of $\Sigma$ defines a complex structure $J$ on $\Sigma$. We denote by $C$ the corresponding Riemann surface. The vortex type equation associated with the continuous parameters $(g,A_0,t)$ is 
$$
\left\{\begin{array}{ccc}
\bar\partial_{A,A_0} \varphi&=&0\ \ \\
i\Lambda_g F_A-\frac{1}{2} \varphi^*\circ\varphi&=&-t\id_L\ . 
\end{array}\right. \eqno{(V_t)}
$$
These are $\mathcal{G}$-equivariant equations for a pair $(A,\varphi)\in \mathcal{A}(M,E_0)$.  

Denote by  $\mathcal{M}_t(M,E_0,A_0)$  the moduli space  of gauge equivalence classes of solutions, and by $\mathcal{M}_t^*(M,E_0,A_0)$ the open subspace of irreducible solutions. Let $Quot_{\mathcal{E}_0}^M$ be the Quot space of quotients of $\mathcal{E}_0$ with  kernel of topological type $M$. This Quot space has the structure of a complex projective variety.

One of the main results  from \cite{OT2} is
\begin{thry} When $t\ne -\frac{2\pi}{\Vol_g(\Sigma)}\deg(M)$,  one has $\mathcal{M}_t(M,E_0,A_0)=\mathcal{M}_t^*(M,E_0,A_0)$ and there is a natural isomorphism
$$\mathcal{M}_t(M,E_0,A_0)\textmap{I}\left\{\begin{array}{ccc}
\emptyset&\rm if& t<-\frac{2\pi}{\Vol_g(\Sigma)}\deg(M)\ \ \\
Quot^M_{\mathcal{E}_0}&\rm if& t>-\frac{2\pi}{\Vol_g(\Sigma)}\deg(M)\ .
\end{array}
\right.
$$
\end{thry}

Now fix a {\it Real structure} on the pair $(C,E_0)$ and suppose that our continuous parameters $(g,A_0)$ are also Real. More precisely, we fix a real structure $\iota$ on $C$ and a $\iota$-Real structure $\tilde\iota_0$ on $E_0$. We require that $g$ is $\iota$-invariant and  $A_0$ is a   Real Hermitian connection  (i.e., the horizontal distribution of $A_0$ is invariant under the diffeomorphism induced by $\tilde\iota_0$ on the principal $\U(r_0)$-bundle $P_{E_0}$ of $E_0$.)  Then $(C,\iota)$ is a Klein surface, and $(\mathcal{E}_0,\tilde\iota_0)$ becomes a Real holomorphic bundle on $(C,\iota)$. Therefore the Quot space $Quot_{\mathcal{E}_0}^M$ inherits the structure of a projective variety defined over $\R$.

\begin{thry} Let $(C,\iota)$ be a Klein surface endowed with an $\iota$-invariant Riemannian metric $g$, $(E_0,\tilde\iota_0)$ a Real Hermitian vector bundle of rank $r_0$, and $M$ a Hermitian line bundle on $C$. Fix  be a Real Hermitian connection  $A_0$ on $E_0$. Then $\iota$ induces an involution 
$$\hat\iota:\mathcal{B}(M,E_0)\to \mathcal{B}(M,E_0)
$$
leaving the moduli space $\mathcal{M}_t(M,E_0,A_0)$ invariant. For $t>-\frac{2\pi}{\Vol_g(\Sigma)}\deg(M)$, the isomorphism $I:\mathcal{M}_t(M,E_0,A_0)\to 
Quot_{\mathcal{E}_0}^M$  is an isomorphism of Real complex spaces. The fixed point locus $\mathcal{M}_t(M,E_0,A_0)^{\hat\iota}$ can therefore be identified with the subspace of real points of the projective variety
$Quot_{\mathcal{E}_0}^M$. 
\end{thry}
\begin{proof} The involution $\hat\iota: \mathcal{B}(M,E_0)\to \mathcal{B}(M,E_0)$ is constructed in the following way. Since line bundles on a Riemann surface are classified by their first Chern class and one obviously has
$$\iota^*(c_1(M))=-c_1(L)=c_1(\bar M))\ ,$$
there exists an isomorphism $\iota^*(M)\simeq \bar  M$, i.e.,  there exists a $\iota$-covering anti-linear bundle isomorphism $f:M\to M$. We put
$$f^*(A,\varphi):=(f^*(A),\tilde\iota_0\circ\varphi\circ f)\ ,
$$
where $f^*(A)$ is the direct image of the connection $A$ under the $\iota$-covering bundle isomorphism $P_M\to P_M$  of type $\stackrel{-}{}\ :\U(1)\to\U(1)$ defined by $f^{-1}$. Since $f\circ f\in\mathcal{G}$ and   
$$f^*(f^*(A,\varphi))=((f\circ f)^*A,\varphi\circ(f\circ f))\ ,$$
 $f^*$ induces an involution $\hat\iota:\mathcal{B}(M,E_0)\to \mathcal{B}(M,E_0)$.  This involution $\hat\iota$   depends only on $\iota$ and not on the choice of $f$, because $f$ is well defined up to right composition with a gauge transformation. Since the second  equation of $(V_t)$ is invariant  and the first equation is  equivariant with respect to $f^*:\mathcal{A}(M,E_0)\to \mathcal{A}(M,E_0)$,  $\hat\iota$ leaves the moduli space $\mathcal{M}_t(M,E_0,A_0)$ invariant as claimed.
Suppose now that $t>-\frac{2\pi}{\Vol_g(\Sigma)}\deg(M)$. The   involution $c:Quot_{\mathcal{E}_0}^M\to Quot_{\mathcal{E}_0}^M$     which defines the natural real structure of this Quot space,  is given as follows: The locally free $\mathcal{O}_C$-sheaf  (associated  with) $\mathcal{E}_0$ is naturally a $\Z_2$-sheaf (in the sense of \cite{G}) on the Klein surface $(C,\iota)$ via the formula
$$\ig_0(s):=\tilde \iota_0\circ s\circ\iota$$
 for  any open set $U\subset C$ and any holomorphic section $s\in \Gamma(U,\mathcal{E}_0)$. Then, for a coherent subsheaf  $\mathcal{K}\subset \mathcal{E}_0$ we have
$$c\left(\left[\mathcal{E}_0\to \qmod{\mathcal{E}_0}{\mathcal{K}}\right]\right)=\left[\mathcal{E}_0\to \qmod{\mathcal{E}_0}{\ig_0(\mathcal{K})}\right]\ .
$$

On the other hand, by the definition of $I$, the quotient associated with a vortex $(A,\varphi)$ is the quotient $[\mathcal{E}_0\to \mathcal{E}_0/\varphi(\mathcal{M}_A)]$, where $\mathcal{M}_A$ is the rank 1 locally free sheaf associated with  the holomorphic structure defined by $\bar\partial_A$ on $M$. It suffices to note that
$$(\tilde\iota_0\circ\varphi\circ f)(\mathcal{M}_{f^*(A)})=\ig_0(\varphi(\mathcal{M}_A))\ .
$$
\end{proof}

We will need a generalization of Lemmata  \ref{ImpRed}, \ref{ImpRe} (and of their Real analogon Remark \ref{RealIso}) for the determinant line bundle 
$$\det\mathrm{ind}\ \delta_{p_0}^{L,\mathcal{E}_0}\simeq \det(R^0q_*(\L_{p_0}\otimes p^*(\mathcal{E}_0)))\otimes\left[\det(R^1q_*(\L_{p_0}\otimes p^*(\mathcal{E}_0))\right)]^\vee   $$
on $\Pic^d(C)$, where $\L_{p_0}$ is the Poincaré line bundle associated with a differentiable line bundle $L$ of degree $d$ on $C$ and a point $p_0\in C$, and  $\mathcal{E}_0$ is a holomorphic vector bundle of rank $r_0$ and degree $e_0$ on $C$. The maps $p$, $q$ are the projections
$$p:\Pic^d(C)\times C\to C\ ,\ q:\Pic^d(C)\times C\to \Pic^d(C)\ .$$
When $r_0=1$, one can use Lemma  \ref{ImpRed}  and the functoriality of the determinant line bundle with respect to isomorphic base changes to see that there exists a canonical isomorphism
$$
\det\mathrm{ind}\ \delta_{p_0}^{L,\mathcal{E}_0}= \tau_{\mathcal{E}_0}^*(\det\mathrm{ind}\ \delta_{p_0}^{L\otimes E_0}) \simeq $$
$${\cal O}_{\Pic^d(C)}(\Theta+[{\cal O}_C((d-g+1)p_0)\otimes\mathcal{E}_0^\vee\otimes\mathcal{O}_C(e_0p_0)])\ .
$$
For the general situation we have
\begin{pr} On has   canonical isomorphisms 
$$\det\mathrm{ind}\ \delta_{p_0}^{L,\mathcal{E}_0}=\mathcal{O}_{\Pic^d(C)}(\Theta+[\mathcal{O}_C((d+1-g)p_0)])^{\otimes (r_0-1)}\otimes$$
$$ \ \ \ \ \ \ \ \ \ \mathcal{O}_{\Pic^d(C)}(\Theta+[\mathcal{O}_C((d+1-g)p_0)\otimes (\det \mathcal{E}_0)^\vee\otimes\mathcal{O}_C(e_0 p_0)]) \ .
$$
When $p_0\in C^\iota$ and $\mathcal{E}_0$ is a Real holomorphic bundle, then  this canonical isomorphism is an isomorphism of Real holomorphic line bundles.
\end{pr}
\begin{proof} The statement has  been proved for $r_0=1$. If $r_0>0$ we proceed by induction with respect to $r_0$. The bundle $\mathcal{E}_0$ fits into a short exact sequence of vector bundles
$$0\to \mathcal{K}_0\to \mathcal{E}_0\to \mathcal{F}_0\to 0\ ,
$$
with   $\mathcal{K}_0$ of rank 1. Hence  one has  
$$\det\mathrm{ind}\ \delta_{p_0}^{L,\mathcal{E}_0}=\det\mathrm{ind}\ \delta_{p_0}^{L,\mathcal{K}_0}\otimes \det\mathrm{ind}\ \delta_{p_0}^{L,\mathcal{F}_0}\ .$$
Using the holomorphic line bundles  
$$\mathcal{L}=\mathcal{O}_{\Pic^{g-1}}(\Theta),\  \mathcal{P}:=\mathcal{O}_C((g-1-d)p_0)  ,\ $$
$$\mathcal{V}:= \mathcal{K}_0\otimes\mathcal{O}_C(-k_0p_0) ,\  \mathcal{W}:=(\det\mathcal{F}_0)\otimes\mathcal{O}_C(-f_0p_0)
$$
with $k_0:=\deg(\mathcal{K}_0)$, $f_0:=\deg(\mathcal{F}_0)$, we can write
$$\det\mathrm{ind}\ \delta_{p_0}^{L,\mathcal{E}_0}= \tau_{\mathcal{P}\otimes\mathcal{V}}^*(\mathcal{L})\otimes \tau_{\mathcal{P}}^*(\mathcal{L})^{\otimes (r_0-2)}\otimes \tau_{\mathcal{P}\otimes\mathcal{W}}^*(\mathcal{L})\ .
$$
It suffices to note that  $\mathcal{V}\otimes \mathcal{W}=(\det\mathcal{E}_0)\otimes\mathcal{O}_C(-e_0p_0)$ and that
$$\tau_{\mathcal{P}\otimes\mathcal{V}}^*(\mathcal{L})\otimes \tau_{\mathcal{P}\otimes\mathcal{W}}^*(\mathcal{L})=\tau_{\mathcal{P}}^*(\tau_{\mathcal{V}}^*(\mathcal{L})\otimes \tau^*_\mathcal{W}(\mathcal{L}) )=\tau_\mathcal{P}^*(\tau_{\mathcal{V}\otimes \mathcal{W}}(\mathcal{L}))\otimes\tau_\mathcal{P}^*(\mathcal{L})\ .
$$
The last equality follows from the theorem of the square (see Theorem 2.3.3 \cite{BL}). The arguments generalize easily to the Real case.
\end{proof}
\begin{re}\label{newdet} In the special  case  $\det \mathcal{E}_0 =\mathcal{O}_C(e_0p_0)$ one obtains canonical isomorphism
$$\det\mathrm{ind}\ \delta_{p_0}^{L,\mathcal{E}_0}=\mathcal{O}_{\Pic^d(C)}(\Theta+[\mathcal{O}_C((d+1-g)p_0)])^{\otimes  r_0}\ ,
$$
$$\tau_{\mathcal{O}_C(dp_0)}^*(\det\mathrm{ind}\ \delta_{p_0}^{L,\mathcal{E}_0})=\mathcal{O}_{\Pic^0(C)}(\Theta+[\mathcal{O}_C((1-g)p_0)])^{\otimes  r_0}\ ,
$$
which are isomorphisms of Real holomorphic line bundles when $(\mathcal{E}_0,\tilde\iota_0)$ is a Real holomorphic vector bundle and $p_0\in C^\iota$.
\end{re}
\vspace{3mm}

Let   $\mathcal{Q}uot^M_{\mathcal{E}_0}(\R):=\{\mathcal{Q}uot^M_{\mathcal{E}_0}\}^c$ be the subspace of real points in the Quot space $\mathcal{Q}uot^M_{\mathcal{E}_0}$. Put $L:=M^\vee$, $d:=\deg(L)$. We will study the   orientability of the components of $ \mathcal{Q}uot^M_{\mathcal{E}_0}(\R)$ and we will show that, at least for sufficiently large $d$,  this problem can be reduced to the computation  of  the first Stiefel-Whitney class of  the  real line bundle associated with the Real holomorphic line bundle  $\det\mathrm{ind}\ \delta^{L,\mathcal{E}_0}_{p_0}$ on $\Pic^{d}(C)$,  or equivalently  its translate $\tau_{\mathcal{O}_C(d p_0)}^*(\det\mathrm{ind}\ \delta^{L,\mathcal{E}_0}_{p_0})$ on $\Pic^0(C)$. 

The Quot space $\mathcal{Q}uot^M_{\mathcal{E}_0}$ can be regarded as a projective fibration over $\Pic^{d}(C)$ via the map $\pi:\mathcal{Q}uot^M_{\mathcal{E}_0}\to\Pic^{d}(C)$ defined by 
$$\pi([q:\mathcal{E}_0\to Q]):=\ker(q)^\vee\ .$$
 Via the identification $I:\mathcal{M}_t(M,E_0,A_0)\to \mathcal{Q}uot^M_{\mathcal{E}_0}$ the map $\pi$ is simply given  by $\pi([A,\varphi]):=[\mathcal{M}_{\bar\partial_A}^\vee]$. The fiber over a point $[{\cal L}]\in\Pic^{d}(C)$ is the projective space $\P(H^0({\cal L}\otimes \mathcal{E}_0))$. Suppose now  that 
\begin{equation}\label{negative} d>\max (-\mu(\mathcal{E}_0)+2(g-1),\mu_{\max}(\mathcal{E}_0)(r_0-1)-\mu(\mathcal{E}_0) r_0 +2(g-1))\ ,\end{equation}
where 
$$\mu_{\max}(\mathcal{E}_0): =\sup\{\mu(\mathcal{F})|\ \mathcal{F}\subset \mathcal{E}_0\hbox{ a nontrivial subsheaf}\} \ .$$
Then $h^1({\cal L}\otimes \mathcal{E}_0)=0$ and  $h^0({\cal L}\otimes \mathcal{E}_0)=(e_0+r_0d)+r_0(1-g)$ for every holomorphic line bundle $\mathcal{L}$ of degree $d$ (see \cite{BDW}). 

  Fixing a Real structure on $L$ (or equivalently on $M$) we obtain  an induced Real structure on the Poincaré line bundle $\L_{p_0}$, hence  $\L_{p_0}\otimes p^*(\mathcal{E}_0)$ and $R^0q_*(\L_{p_0}\otimes p^*(\mathcal{E}_0))$ become   Real  holomorphic bundles  on   $\Pic^{d}(C)\times C$ and $\Pic^{d}(C)$ respectively.

We identify $\mathcal{Q}uot^M_{\mathcal{E}_0}$ with the projectivization   $\P(\E_0)$ of the holomorphic  bundle $\E_0$  on $\Pic^{d}(C)$ which is associated with the locally free sheaf  $R^0q_*(\L_{p_0}\otimes p^*(\mathcal{E}_0))$.  
 It is easy to see that the involution $c:{Q}uot^M_{\mathcal{E}_0}\to {Q}uot^M_{\mathcal{E}_0}$ is induced by the Real structure $\ig_0$ of $\E_0$. The fixed point bundle $\F_0:=\E_0^{\ig_0}$ is a real vector  bundle of rank $(e_0+r_0 d)+r_0(1-g)$ over $\Pic^d(C)^{\hat\iota}$, and one obtains a natural identification $\mathcal{Q}uot^M_{\mathcal{E}_0}(\R) \simeq\P_\R(\F_0)$. The relative Euler  sequence associated with the projective fiber bundle $q_0:\P_\R(\F_0)\to\Pic^{d}(C)^{\hat\iota}$ reads
$$0\to \P_\R(\F_0)\times\R\to q_0^*(\F_0)(\lambda^\vee)\to T_{q_0}\to 0 \ .
$$
Here $T_{q_0}\subset T_{\P_\R(\F_0)}$ stands for the vertical tangent  bundle of $q_0$, and $\lambda$ denotes the tautological line bundle on $\P_\R(\F_0)$. Note that all connected components  of   
$\Pic^{d}(C)^{\hat\iota}$ are tori, hence they are orientable. Therefore
\begin{equation}\label{w_1}
w_1(T_{\P_\R(\F_0)})=w_1(T_{q_0})=w_1(q_0^*(\F_0)(\lambda^\vee))=$$
$$=q_0^*(w_1(\F_0))+((e_0+r_0d)+r_0(1-g))w_1(\lambda^\vee) \ .
\end{equation}
 Note that  $w_1(\lambda^\vee) \not\in q_0^*(H^1(\Pic^{-d}(C)^{\hat\iota},\Z_2))$ for $(e_0+r_0d)+r_0(1-g)>1$.

On the other hand
$$w_1(\F_0)=w_1(\wedge^{d+1-g} (\F_0))=w_1((\det\mathrm{ind}\ \delta_{p_0}^{L,\mathcal{E}_0})^{\tilde{\hat\iota}})\ ,
$$
where $\tilde{\hat\iota}$ denotes the canonical $\hat\iota$-Real structure on the determinant line bundle $\det\mathrm{ind}\ \delta_{p_0}^{L,\mathcal{E}_0}$, given by Remark \ref{RealStr1}.   Using formula (\ref{w_1})  we obtain:
\begin{pr} \label{OrSymm} Let $(C,\iota)$ be a Klein surface, $(\mathcal{E}_0,\tilde\iota_0)$ a  Real holomorphic vector bundle of rank $r_0$ on $C$, $p_0\in C^\iota$  a fixed point of $\iota$, and  suppose $d\in\Z$ satisfies   $d>-\mu(\mathcal{E}_0)+(g-1)+\frac{1}{r_0}$  and (\ref{negative}). Regard  $\mathcal{Q}uot^M_{\mathcal{E}_0}(\R)$ as a real projective bundle over $\Pic^{d}(C)^{\hat\iota}$ via $\pi$. Let $T\subset \Pic^d(C)^{\hat\iota}$ be a connected component of $\Pic^d(C)^{\hat\iota}$ and $\{\mathcal{Q}uot^M_{\mathcal{E}_0}(\R)\}_T$ the corresponding component of $\mathcal{Q}uot^M_{\mathcal{E}_0}(\R)$. Then $\{\mathcal{Q}uot^M_{\mathcal{E}_0}(\R)\}_T$ is orientable if and only if $(e_0+r_0d)+r_0(1-g)$ is even and 
$$w_1\left(\resto{(\det\mathrm{ind}\ \delta_{p_0}^{L,\mathcal{E}_0})^{\tilde{\hat\iota}}}{T}\right)=0\ .
$$
\end{pr}
Note that the last condition depends effectively on the component $T$ (see Proposition \ref{difference}, section \ref{RYMT}). 

Suppose for simplicity that $\det(\mathcal{E}_0)\simeq \mathcal{O}_C(e_0 p_0)$ as Real holomorphic line bundles. In this case, using the Remarks \ref{RealIso}, \ref{newdet}   the computation of the Stiefel-Whitney class  $w_1\left(\resto{(\det\mathrm{ind}\ \delta_{p_0}^{L,\mathcal{E}_0})^{\tilde{\hat\iota}}}{T}\right)$ can be reduced to   computing  the Stiefel-Whitney class of the   real  bundle associated with ${\cal O}_{\Pic^0(C)}\left(\Theta-[{\cal O}_C((g-1)p_0)]\right)$ on the real torus $\tau_{ {\cal O}_C(dp_0)}^*(T)\subset\Pic^0(C)^{\hat\iota}$.

This gives a clear geometric motivation for the computation of the first Stiefel-Whitney class  $w_1\left( {\cal O}_{\Pic^0(C)}\left(\Theta-[{\cal O}_C((g-1)p_0)]\right)^{\tilde{\hat\iota}} \right)$. We will come back to this orientability problem in section \ref{results} (see Proposition \ref{Beispiel}).
\begin{re} In the special case  when $\mathcal{E}_0=\mathcal{O}_C$, the Quot space $\mathcal{Q}uot^{L^\vee}_{\mathcal{E}_0}$ can be identified with the $d$-th symmetric power $S^d(C)$  as a Real space.
\end{re}

\section{Real Hermitian  line bundles}

\subsection{Grothendieck's formalism}

Let $X$ be a paracompact space endowed with an involution $\iota$. Regard $X$ as a $\Z_2$-space, and denote by  $\underline{S}^1$ (respectively ${\underline{S}^1}(1)$) the $\Z_2$-sheaf on $X$ of $S^1$-valued smooth functions, with the  $\Z_2$-action defined by composition with $\iota$ (respectively by composition with $\iota$ and conjugation). 

We recall from \cite{G} the following classification theorem for equivariant principal bundles.
\begin{pr} Let $(X,\gamma)$, $\gamma:\Gamma\times X\to X$  be a paracompact $\Gamma$-space, where $\Gamma$ is a finite group, and let $G$ be a Lie group endowed with a group morphism $\alpha:\Gamma\to \Aut(G)$.  Then there is a canonical bijection
$$\{\hbox{Iso classes of $\alpha$-equivariant principal $G$-bundles}\}\simeq H^1_\Gamma(X;\underline{G}(\alpha))\ ,
$$
where $\underline{G}(\alpha)$ stands for the $\Gamma$-sheaf of continuous $G$-valued maps on $X$ with $\Gamma$-action defined   via $\alpha$.
\end{pr}
\begin{re} When $X$ is a differentiable manifold one obtains a similar result replacing the $\Gamma$-sheaf of continuous $G$-valued maps by the $\Gamma$-sheaf of smooth $G$-valued maps on $X$. Moreover the cohomology sets associated with the two sheaves can be identified as in the non-equivariant case. 
\end{re}
This can be seen be comparing the standard spectral sequences associated with the two sheaves at the $E_2$-level.\\

For any Abelian group $A$ one has two obvious $\Z_2$-actions on $A$: the trivial action $\alpha_0$  and the inversion action $\alpha_1$. We agree to write $(0)$ and $(1)$ for  the twistings by $\alpha_0$ and $\alpha_1$, and we agree to omit $(0)$. Let $^-$ be the conjugation action of $\Z_2$ on $\C$ and $\C^*$.
\begin{re} Let $(X,\iota)$ be a space with involution. The set of  isomorphism classes of $\iota$-Real line bundles on $X$ can be identified with the set of isomorphism classes of  Hermitian $\iota$-Real line bundles on $X$. More precisely the monomorphism $\underline{S}^1(1)=\underline{S}^1(^-)\to\underline{\C}^*(^-)$  induces isomorphism in positive cohomology.
\end{re}
Indeed, it suffices to see  that $H^k_{\Z_2}(X,\underline{\R})=0$  for any $k>0$  (using again  the standard spectral sequence associated with this sheaf). \\

In particular, we obtain an identification
$$\{[L,\tilde\iota]\ \vline\ (L,\tilde\iota) \hbox{ Real line bundle over }(X,\iota)  \}=H^1_{\Z_2}(X,\underline{S}^1(1))\ .
$$

Denote by ${\cal G}(1)$  the $\Z_2$-module structure on the gauge group ${\cal G}$  defined by the involution $g\mapsto \iota^*(\bar g)$. Using the standard spectral exact sequence associated with the $\Z_2$-sheaf $\underline{S}^1(1)$     one obtains an exact sequence
\begin{equation}\label{E2}  H^1_{\Z_2}(H^0(X,\underline{S^1}(1)))\map H^1_{\Z_2}(X,\underline{S}^1(1))\map  H^0_{\Z_2}(H^1(X,\underline{S}^1(1)))=$$
$$=H^0_{\Z_2}(H^2(X,\Z)(1))\textmap{d_2}  H^2_{\Z_2}(H^0(X,\underline{S}^1(1)))\ .
\end{equation}
The $\Z_2$-module $H^0(X,\underline{S^1}(1))$ in the first and in the last term above is just the gauge group ${\cal G}$ regarded as a $\Z_2$-module via the involution $g\mapsto \iota^*(\bar g))$. We use the notation ${\cal G}$ for the $\Z_2$-module structure  defined by the involution $g\mapsto \iota^*(g)$ (not the trivial $\Z_2$-module structure!). The first cohomology group $H^1_{\Z_2}({\cal G}(1))$ of ${\cal G}(1)$ fits into the short exact sequence of $\Z_2$-modules
$$0\map {\cal G}(1)^{\Z_2}\map {\cal G}\textmap{\Sigma} {\cal G}^{\Z_2}\map H^1_{\Z_2}({\cal G}(1))\map 0\ ,
$$
where $\Sigma:{\cal G}\to {\cal G}^{\Z_2}$ is the morphism $g\mapsto g\iota^*(g)$. This proves the following 
\begin{pr}\label{firstES} One has an exact sequence

$$\hspace{-3mm}1\to{\cal G}(1)^{\Z_2}\to{\cal G}\textmap{\Sigma}{\cal G}^{\Z_2}\textmap{\lambda}H^1_{\Z_2}(X,\underline{S}^1(1))\stackrel{c_1}{\to} H^2(X,\Z)(1)^{\Z_2}\stackrel{\ooo}{\to}H^2_{\Z_2}({\cal G}(1)),
$$
where 
\begin{enumerate}
\item The morphism $\Sigma$ is given by $\Sigma(g):=g(\iota^*g)$, 
\item $\ker(c_1)=\qmod{{\cal G}^{\Z_2}}{\Sigma({\cal G})}=H^1_{\Z_2}({\cal G}(1))$, 
\item  $H^2(X,\Z)(1)^{\Z_2}=\{x\in H^2(M,\Z)|\ \iota^*(x)=-x\}$.

\end{enumerate}
\end{pr}

To compute $H^1_{\Z_2}({\cal G}(1))$ we use  the short exact sequence  of $\Z_2$-modules 
\begin{equation}\label{ses} 0\map {\cal G}_0(1)\map {\cal G}(1)\map H^1(X,\Z)(1)\map 0\ , 
\end{equation}
where ${\cal G}_0={\cal C}^\infty(X,\R)/\Z$ is the connected component of the identity in ${\cal G}$.  This connected component fits into the short exact sequence
\begin{equation}\label{ss}0\map  \Z(1)\map {\cal C}^\infty(X,\R)(1)\map {\cal G}_0(1)\map 0\ .\end{equation}
One has $H^k_{\Z_2}({\cal C}^\infty(X,\R)(1))=0$ for all $k\geq 1$.  Therefore 
$$H^{2k-1}_{\Z_2}({\cal G}_0(1))=H^{2k}_{\Z_2}(\Z(1))=0\ ,\ H^{2k}_{\Z_2}({\cal G}_0(1))=H^{2k+1}_{\Z_2}(\Z(1))=\Z_2\ ,\ \forall k\geq 1\ .$$
We get an exact sequence
$$0\to H^1_{\Z_2}({\cal G}(1))\to H^1_{\Z_2}(H^1(X,\Z)(1))\to H^2_{\Z_2}({\cal G}_0(1))\to H^2_{\Z_2}({\cal G}(1))\to $$
$$\to H^2_{\Z_2}(H^1(X,\Z)(1))\to 0 \ .
$$

When $X^\iota\ne \emptyset$, we choose $x_0\in X^\iota$ and we notice that the composition of the maps
$$S^1(1)\map {\cal G}_0(1)\map {\cal G}(1)\textmap{\mathrm{ev}_{x_0}}S^1(1)
$$
is the identity, and that the first map $S^1(1)\map {\cal G}_0(1)$ induces an isomorphism in cohomology groups of strictly positive degree. This is so since $S^1(1)$ fits into the short exact sequence $1\to\Z(1)\to\R(1)\to S^1(1)\to 1$, which can be easily compared to (\ref{ss}). Therefore the morphism $H^2_{\Z_2}({\cal G}_0(1))\to H^2_{\Z_2}({\cal G}(1))$ is injective, and we get
\begin{re} \label{rem} When $X^\iota\ne \emptyset$, the natural map  
$$H^1_{\Z_2}({\cal G}(1))\map H^1_{\Z_2}(H^1(X,\Z)(1))  $$
is an isomorphism, and one has the short exact sequence
$$
0\to  H^2_{\Z_2}({\cal G}_0(1))\simeq \Z_2\textmap{j} H^2_{\Z_2}({\cal G}(1))\textmap{q} H^2_{\Z_2}(H^1(X,\Z)(1))\map 0\ .
$$
The generator of 
$$H^2_{\Z_2}({\cal G}_0(1))=\qmod{{\cal G_0}(1)^{\Z_2}}{\{g\iota^*(\bar g)|\ g\in{\cal G}_0\}}$$
 is the class modulo $\{g\iota^*(\bar g)|\ g\in{\cal G}_0\}$ of the constant gauge transformation $-1\in S^1$.
\end{re}
 
\begin{lm} Suppose $X^\iota\ne\emptyset$. The edge morphism 
$$d_2:  H^0_{\Z_2}(H^2(X,\Z)(1))=H^2(X,\Z)(1)^{\Z_2}\map   H^2_{\Z_2}({\cal G}(1))$$
has the property $\ker(d_2)=\ker(q\circ d_2)$.
\end{lm} 
\begin{proof}
 Indeed, if $c\in H^2(X,\Z)(1)^{\Z_2}$ belongs to $\ker(q\circ d_2)$, then we get  $d_2(c)\in j(H^2_{\Z_2}({\cal G}_0(1)))$. It suffices to notice that $d_2(c)$ can never coincide with the class $[-1]$ modulo $\{g\iota^*(\bar g)|\ g\in{\cal G}_0\}$. This can be seen as follows:

The morphism $d_2: H^2(X,\Z)(1)^{\Z_2}\to H^2_{\Z_2}({\cal G}(1))$ can be geometrically interpreted as follows: consider a Hermitian line bundle $L$ on $X$ with Chern class $c_1(L)=c\in H^2(X,\Z)(1)^{\Z_2}$. Since $\iota^*(c)=-c$, it follows that $\iota^*(L)\simeq\bar L$, so there exists an anti-linear isometry $\sigma:L\to \iota^*(L)$.  We get a smoothly varying family of anti-linear isometries  $\sigma_x:L_x\to L_{\iota(x)}$. The composition $\phi_x=\sigma_{\iota_x}\circ\sigma_x:L_x\to L_x$ is $\C$-linear, so it can be regarded as an element in $S^1$, depending smoothly on $x\in X$. It is easy to see that 
$\iota^*(\phi)=\bar\phi$. The element $d_2(c)$ is just the class $[\phi]$ modulo the subgroup $\{g\iota^*(\bar g)|\ g\in{\cal G}_0\}$. We have to show that $[\phi]\ne[-1]$. Choose $x_0\in X^\iota$ and a  unitary identification $L_{x_0}\simeq \C$. The anti-linear isometry $\sigma_{x_0}$ acts as $\sigma_{x_0}(l)=\zeta\bar l$, for a constant $\zeta\in S^1$. Therefore $\phi_{x_0}=\zeta\bar \zeta=1$. If   $[\phi]=[-1]$, one would have $\phi_{x_0}=-\psi_{x_0}\bar\psi_{x_0}=-1$ for a smooth $S^1$-valued function $\psi$, which yields obviously a contradiction.
\end{proof}

Using  Proposition \ref{firstES}  and Remark  \ref{rem} we obtain:

\begin{co} \label{RealLB} Suppose $X^\iota\ne\emptyset$. There exists an exact sequence
$$0\to H^1_{\Z_2}(H^1(X,\Z)(1))\to H^1_{\Z_2}(X,\underline{S}^1(1))\textmap{c_1} H^2(X,\Z)(1)^{\Z_2}\textmap{\ooo} H^2_{\Z_2}(H^1(X,\Z)(1))\ .
$$
\end{co}
\vspace{3mm}
As in the classical classification theory for vector bundles, it is important to give an explicit description of the set $H^1_{\Z_2}(X,\underline{S}^1(1))$ of isomorphism classes of $\iota$-Real line bundles on $X$ in terms of characteristic classes.  The relevant characteristic classes associated to a $\iota$-Real Hermitian line bundle $(L,\tilde\iota)$  on $(X,\iota)$ are: 
$$c_1(L)\in \ker(\oo)\subset H^2(X,\Z)(1)^{\Z_2}\ ,\ w_1(L^{\tilde\iota})\in H^1(X^\iota,\Z_2)\ .$$

Therefore it is a natural problem to determine explicitly the kernel and the image of the corresponding group morphism
$$H^1_{\Z_2}(X,\underline{S}^1(1))\textmap{cw} H^2(X,\Z)(1)^{\Z_2}\times H^1(X^\iota,\Z_2)\ .
$$
Determining these groups will give an alternative short exact sequence having  the group $H^1_{\Z_2}(X,\underline{S}^1(1))$ as central term.

\subsection{Examples}\label{examples}

In this section we will apply the general formalism developed  above in two important cases:    Klein surfaces  and a Real tori.
\vspace{3mm}

\paragraph{\bf The case of a  Klein  surface}   Let $C$ be a closed connected, oriented  differentiable 2-manifold, and $\iota:C\to C$ an orientation reversing involution with $C^\iota\ne\emptyset$.  Let $r\in\N$ be the number of components of $C^\iota$.
Let $d_2:H^2(C,\Z)\to\Z_2$,  $\deg_{\Z_2}: H^1(C^\iota,\Z_2)\to\Z_2$  be the morphisms defined by 
$$\alpha\mapsto \langle\alpha,[C]\rangle\hbox { mod 2 },\  \gamma\mapsto \langle \gamma,[C^\iota]_{\Z_2}\rangle\ ,$$
 and denote by $H^2(C,\Z)(1)^{\Z_2}\times_{\Z_2} H^1(C^\iota,\Z_2)$ the fiber product of $d_2$ and $\deg_{\Z_2}$.
\begin{thry} \label{ClassOnKlein} The characteristic map  
$$cw: H^1_{\Z_2}(C,\underline{S}^1(1))\map  H^2(C,\Z)(1)^{\Z_2}\times H^1(C^\iota,\Z_2)$$
induces an isomorphism
$$H^1_{\Z_2}(C,\underline{S}^1(1))\textmap{cw} H^2(C,\Z)(1)^{\Z_2}\times_{\Z_2} H^1(C^\iota,\Z_2) \ .
$$

\end{thry}
\begin{proof}
 
It follows from the Corollary in  Appendix B  that  for any $\iota$-Real line bundle $(L,\tilde\iota)$ one has 
$$\deg_{\Z_2}(w_1(L^{\tilde\iota}))\equiv \langle c_1(L),[C]\rangle \hbox { (mod 2) } \ .$$

 This shows that $\im(cw)\subset H^2(C,\Z) \times_{\Z_2} H^1(C^\tau,\Z_2)$, and  that we have a commutative  diagram:
$$\hspace{-3mm}\begin{array}{ccccccccc}
0&\to &H^1_{\Z_2}(H^1(C,\Z)(1))&\to& H^1_{\Z_2}(C,\underline{S}^1(1))&\textmap{c_1}& H^2(C,\Z)&\to&0\\
&&\downarrow \jg&&\downarrow cw&&\|
\\
0&\to&\ker(\deg_{\Z_2})&\to&H^2(C,\Z) \times_{\Z_2} H^1(C^\iota,\Z_2)&\textmap{\mathrm{pr}_1}& H^2(C,\Z)&\to&0
\end{array}
$$
One easily sees that $cw$ is surjective e.g. by choosing a $\iota$-anti-invariant holomorphic structure $J$ on $C$ and using Real divisors to construct real line bundles with prescribed characteristic classes.

Using a Comessatti basis for the $\Z_2$-module $H^1(C,\Z)$ one computes 
$$H^1_{\Z_2}(H^1(C,\Z)(1))\simeq \Z_2^{r-1}\ .$$
 On the other hand one obviously  has $\ker(\deg_{\Z_2})\simeq \Z_2^{r-1}$. The claim follows now from the snake lemma.
\end{proof}
\vspace{3mm} 
\paragraph{\bf The case of a  torus} 
Let now  $T=V/\Lambda$ be a torus, where $V$ is a real vector space, and $\Lambda\subset V$ is a lattice. Let $\tau:\Lambda\to\Lambda$ be a linear involution, and denote by the same symbol the induced  automorphisms of $V$ and  $T$.

In order to describe the fixed point locus $T^\tau$ we use the short exact sequence of $\Z_2$-modules
$$0\map \Lambda \map V \map T \map 0\ .
$$
The corresponding long exact sequence of cohomology groups reads:
$$0\map \Lambda^\tau\map V^\tau\map T^\tau\map H^1_{\Z_2}(\Lambda)\map 0\ .
$$
This shows that $T^\tau$ decomposes as a disjoint union
$$T^\tau=\union_{[\mu]\in H^1_{\Z_2}(\Lambda)} T_{[\mu]}\ ,
$$
where every connected component $T_{[\mu]}$ is a torus isomorphic to the quotient 
$$T_0:=\qmod{V^\tau}{\Lambda^\tau}\ .$$
Hence $H_1(T_{[\mu]},\Z)=\Lambda^\tau$ and $H^1(T_{[\mu]},\Z)=[\Lambda^\tau]^\vee$.

We are interested in  $\tau$-Real line bundles on $T$. Since $H^1(T,\Z)=\Lambda^\vee$, one gets from Corollary \ref{RealLB} the exact sequence
$$0\to H^1_{\Z_2}(\Lambda^\vee(1))\to H^1_{\Z_2}(T,\underline{S}^1(1))\textmap{c_1} H^2(T,\Z)(1)^{\Z_2}\textmap{\ooo} H^2_{\Z_2}(\Lambda^\vee(1))\ .
$$
One has natural identifications
$$H^1_{\Z_2}(\Lambda^\vee(1))=\qmod{[\Lambda^\vee]^{\tau^*}}{(\id+\tau^*)\Lambda^\vee}\ ,\  H^2(T,\Z)(1)^{\Z_2}= \wedge^2\Lambda^\vee(1)^{\Z_2}\ ,$$
$$ H^2_{\Z_2}(\Lambda^\vee(1))=H^1_{\Z_2} (\Lambda^\vee)=\qmod{[\Lambda^\vee]^{-\tau^*}}{(\id-\tau^*)\Lambda^\vee}\ .
$$
 
We will see that  on a torus  the obstruction map $\oo$ vanishes (see Proposition \ref{exRealstr}). 
For every $\tau$-Real line  bundle $(L,\tilde\tau)$ on $T$, we have an associated Stiefel-Whitney class of the   fixed point bundle $L^{\tilde\iota}$ on $T^\tau$, which is an element $w(L,\tilde\tau)\in H^1(T^\tau,\Z_2)$. Such an element can be regarded as a map
$$w(L,\tilde\tau): \qmod{\Lambda^{-\tau}}{(\id-\tau)\Lambda}\to\Hom(\Lambda^\tau,\Z_2)\ .
$$

In section \ref{RLBT} we will give an explicit description of the group $H^1_{\Z_2}(T,\underline{S}^1(1))$ of isomorphism classes of $\tau$-Real Hermitian line bundles in terms of characteristic classes. We will show that the group morphism
$$H^1_{\Z_2}(T,\underline{S}^1(1))\textmap{cw} H^2(T,\Z)(1)^{\Z_2}\times  H^1(T^\tau,\Z_2)=$$
$$=\wedge^2\Lambda^\vee(1)^{\Z_2}\times\mathrm{Map}\left(\qmod{\Lambda^{-\tau}}{(\id-\tau)\Lambda}\ ,\ \Hom(\Lambda^\tau,\Z_2)\right)
$$
is injective, and we will determine explicitly its image. In particular we will describe the set of maps 
$$
\qmod{\Lambda^{-\tau}}{(\id-\tau)\Lambda}\map \Hom(\Lambda^\tau,\Z_2) 
$$
which correspond to $\tau$-Real line bundles on $T$, i.e., which have the form $w(L,\tilde\tau)$  for a $\tau$-Real line bundle $(L,\tilde\tau)$ on $T$.
\subsection{Real line bundles and connections}
\begin{pr} \label{YMI} Let $(X,\iota)$  be a differentiable manifold endowed with an involution. Let $L$ be a Hermitian line bundle on  $X$ whose Chern class satisfies $\iota^*(c_1(L))=-c_1(L)$, and let ${\cal B}(L)$ be the moduli space of Hermitian connections on $L$. Then
\begin{enumerate}
\item $\iota$ induces a well defined involution $\hat \iota:{\cal B}(L)\to{\cal B}(L)$.
\item Suppose   $X^\iota\ne\emptyset$. The following conditions are equivalent:
\begin{enumerate}
\item ${\cal B}(L)^{\hat \iota}\ne\emptyset$.
\item $L$ admits $\iota$-Real  structures.
\end{enumerate}
\item If one of the two equivalent conditions above is satisfied, then the set of isomorphism classes of $\iota$-Real structures on $L$ can be identified with $\pi_0({\cal B}(L)^{\hat\iota})$. 
\end{enumerate}
\end{pr}
\begin{proof} (1)  We denote by $c:S^1\to S^1$ the conjugation automorphism. Fix a   $\iota$-covering  anti-isomorphism   $f:L\to L$, and denote by the same symbol the induced $\iota$-covering  type $c$-isomorphism   $P_L\to P_L$ between associated principal bundles.  

For a connection $A$ on $P_L$ we define
$$\hat\iota([A]):=[f^*(A)]\ ,
$$
where $f^*(A)$ is the  pull-back connection in the sense of [KN]. In terms of connection forms one has
$$\theta_{f^*(A)}=-f^*(\theta_A)\ .
$$
Since $f$ is well defined up to composition which a gauge transformation, it follows that $\hat\iota$ is well-defined. Since $f\circ f$ is a gauge transformation, it follows that $\hat\iota$ is an involution as claimed. 
\\
\\
(2), (3) Let ${\cal R}$ be the space of $\iota$-Real structures on $L$. The gauge group ${\cal G}$ acts on ${\cal R}$ by conjugation, and the set of isomorphism classes of $\iota$-Real structures on $L$ is  the quotient ${\cal R}/{\cal G}$.

In order to prove (2) and (3) it suffices to construct a surjective map 
$$F:{\cal B}(L)^{\hat\iota}\to {\cal R}/{\cal G}$$
 whose fibers are the connected components ${\cal B}(L)^{\hat\iota}$. Let $A\in{\cal A}(L)$ such that $[A]\in {\cal B}^{\hat\iota}$. It follows that there exists a gauge transformation $g\in{\cal G}$ such that $A=g^*f^*(A)$, in other words, $A=\tilde\iota^*_A(A)$ where $\tilde\iota_A:=f\circ g$ is a type-$c$ $\iota$-covering isomorphism.  Note that $\tilde\iota_A$  is well defined up to multiplication with constant elements $\zeta\in S^1$. The composition $\tilde\iota_A\circ \tilde\iota_A$ is an $A$-parallel gauge transformation, so it is a constant gauge transformation $z\id_L$. Let $x\in X^\iota$ and $v\in L_x$. One has
$$(\tilde \iota_A\circ\tilde \iota_A\circ \tilde\iota_A)(v)=\tilde \iota_A((\tilde \iota_A\circ \tilde\iota_A)(v))=\tilde\iota_A(zv)=\bar z \tilde\iota_A(v)=(\tilde \iota_A\circ \tilde \iota_A)( \tilde\iota_A(v))=z \tilde\iota_A(v)\ ,
$$
hence $z\in \R\cap S^1=\{\pm 1\}$. But  $L_x$  is a  complex line, so it does not admit any anti-linear automorphism with square $-\id_{L_x}$. This shows that $z=1$, so $\tilde\iota_A$ is an involution. Since it is also $\iota$-covering and anti-linear, we get  $\tilde\iota_A\in{\cal R}$.  Replacing $A$ by a gauge equivalent connection  produces an $\iota$-Real structure  which is conjugate to $\tilde\iota_A$, so that we get a well defined element $F([A]):=[\tilde \iota_A]$. 

To see that the map $F$ is surjective  note that a $\iota$-Real structure $\tilde\iota$ defines an involution on the affine space ${\cal A}(L)$. But any  involution on an affine space   has fixed points. For an $\tilde \iota$-invariant connection $A$ one gets obviously $F([A])=[\tilde\iota]$.

It is easy to see that $F$ is continuous with respect to the quotient topologies. Indeed, the $\iota$-real structure $\tilde\iota$ associated with $A$ is $\tilde\iota=f\circ g$, where
$$g^{-1}dg=A-f^*(A)\ .
$$
This shows that the class $[g]\in{\cal G}/S^1$ depends continuously on $A\in{\cal A}(L)$. On the other hand, the equivalence class $[f\circ g]\in{\cal R}/{\cal G}$ depends only on $[g]\in{\cal G}/S^1$, so it depends continuously on $A$ as claimed. 
 
 Since the quotient topology on ${\cal R}/{\cal G}$ is discrete,   $F$ is constant on the connected components of ${\cal B}(L)^{\hat\iota}$.
 
 It remains to prove that the fibers of $F$ are  connected. Let $A$, $B\in{\cal A}(L)$  such that $F([A])=F([B])$. It follows that there exists $g$, $h$, $k\in{\cal G}$ such that
 $$A=(f\circ g)^*(A)\ ,\ B=(f\circ h)^*(B)\ ,\ f\circ g=k\circ(f\circ h)\circ k^{-1}\ ,
 $$
which implies $ k^*(A)=(f\circ h)^*\circ k^*(A)$. Therefore the connections $k^*(A)\sim A$ and $B$ are both $\tilde \iota$-invariant, where $\tilde \iota$ is the $\iota$-Real structure $f\circ h$. But it is easy to see that the space of $\tilde \iota$-invariant connections in ${\cal A}(L)$ is an affine subspace with model linear space $iA^1(X,\R)^{-\iota^*}$, hence this space is connected as claimed.
 \end{proof}
Suppose now that $X$ is a closed manifold endowed with a $\iota$-invariant Riemannian metric $g$.  A  statement similar to the one  above holds when one replaces the infinite dimensional space ${\cal B}(L)$ with the  moduli space ${\cal T}(L)$ of Yang-Mills connections on $L$, which is isomorphic to a $b_1(X)$-dimensional torus. Note that the inclusion map $${\cal T}(L)\hookrightarrow {\cal B}(L)$$  is a homotopy equivalence.  

 \begin{co}  \label{YMT} Let $(X,\iota)$  be a closed Riemannian manifold endowed with an isometric involution.  Let $L$ be a Hermitian line bundle on  $X$ whose Chern class satisfies $\iota^*(c_1(L))=-c_1(L)$, and let ${\cal T}(L)$ be the moduli space of Yang-Mills connections on $L$. Then
\begin{enumerate}
\item $\iota$ induces a well defined involution $\hat \iota:{\cal T}(L)\to{\cal T}(L)$.
\item Suppose   $X^\iota\ne\emptyset$. The following conditions are equivalent:
\begin{enumerate}
\item ${\cal T}(L)^{\hat \iota}\ne\emptyset$.
\item $L$ admits $\iota$-Real  structures.
\end{enumerate}
\item If one of the two equivalent conditions above is satisfied, then the set of isomorphism classes of $\iota$-Real structures on $L$ can be identified with $\pi_0({\cal T}(L)^{\hat\iota})$. 
\end{enumerate}
 \end{co}

 It is useful to  consider the disjoint union of all Yang-Mills tori
 $${\cal T}_X:=\coprod_{  c\in H^2(X,\Z)}{\cal T}(L_c)\ ,$$
 where $L_c$ denotes a Hermitian line bundle with Chern class $c$. This union comes with a well defined involution $\hat \iota$ defined  as the composition of the usual pull-back of connections:
 $${\cal A}(L_c)\ni A\mapsto  \iota^*(A)\in {\cal A}(\iota^*( L_c)) 
 $$
with the canonical  identification ${\cal A}(L)={\cal A}(\bar L)$ induced by  the equality between the total spaces of the principal  bundles $P_L$, $P_{\bar L}$.  Using Corollary \ref{YMT} we obtain

\begin{pr}\label{YM-Real}  Under the conditions and with the notations of Corollary  \ref{YMT} the assignment  $A\mapsto [\tilde\iota_A]$ defines a group morphism 
 $$F_X: {\cal T}_X^{\hat \iota}\map  H^1_{\Z_2} (X,\underline{S}^1(1)) $$
 from the fixed point locus  ${\cal T}_X^{\hat \iota}$ to the group of isomorphism classes of $\iota$-Real Hermitian line bundles  on $X$. This morphism induces an  isomorphism
 $$f_X:\pi_0( {\cal T}_X^{\hat \iota})\textmap{\simeq}  H^1_{\Z_2} (X,\underline{S}^1(1))\ .
 $$
 \end{pr}

This result has an important complex geometric analogon:
\begin{pr} \label{FC} Let $X$ be a compact connected  complex manifold endowed with an anti-holomorphic involution $\iota:X\to X$. Suppose that $X^\iota\ne\emptyset$.  Consider the  induced anti-holomorphic involution $\hat\iota:\Pic(X)\to\Pic(X)$ defined by
$$\hat\iota([{\cal L}]):=[{\iota^*(\overline{\cal L})}]\ .$$
Let ${\cal L}$ be a holomorphic line bundle on $X$ with $[{\cal L}]\in \Pic(X)^{\hat\iota}$. Then 
\begin{enumerate}
\item There exists an anti-holomorphic $\iota$-Real structure $\tilde\iota_{\cal L}$ on ${\cal L}$, which is unique up to multiplication with constant elements $\zeta\in S^1$. 
\item  The assignment  $[{\cal L}]\mapsto [({\cal L},\tilde\iota_{\cal L})]$ defines a group morphism 
$$\Fg_X:\Pic(X)^{\hat\iota}\map H^1_{\Z_2}(X,\underline{S}^1(1))
$$
which maps  the fixed point locus  $\Pic(X)^{\hat\iota}$ to the set of isomorphism classes of $\iota$-Real Hermitian line bundles.
\item   The induced map $\fg_X:\pi_0(\Pic(X)^{\hat\iota})\to H^1_{\Z_2}(X,\underline{S}^1(1))$ defines a bijection between  $\pi_0(\Pic(X)^{\hat\iota})$ and the set of isomorphism classes of $\iota$-Real line bundles with Chern class  in the Neron-Severi group $\NS(X)$ of $X$.
\end{enumerate}
\end{pr}
\begin{proof} 
(1), (2) The construction of the $\iota$-Real structure $\tilde\iota_{\cal L}$ is similar to the construction of  the $\iota$-Real structure $\tilde\iota_A$ in Proposition \ref{YMI}. \\ \\
(3)  If $X$ is Kählerian, the statement follows from Corollary \ref{YM-Real}. Indeed, on Kählerian manifolds the space of Hermite-Einstein connections coincides with the space of  integrable Yang-Mills connections. Therefore, using the Kobayashi-Hitchin correspondence for line bundles \cite{LT}, we see that   $\Pic(X)$ can be identified  with the (open and closed) subgroup of the Yang-Mills group ${\cal T}_X$ consisting  of gauge equivalence classes of Yang-Mills connections on line bundles with Chern class of type (1,1).

For the non-Kählerian case, one has  to replace the Yang-Mills group  ${\cal T}_X$ with the group ${\cal T}_X^{HE}$ of gauge-equivalence classes of Hermite-Einstein connections, and to see that the analogue of Corollary \ref{YM-Real} holds, giving a bijection 
$$\pi_0(\Pic(X)^{\hat\iota})=\pi_0(({\cal T}_X^{HE})^{\hat\iota})\to \{\gamma\in H^1_{\Z_2}(\underline{S}^1(1))|\ c_1(\gamma)\in \NS(X)\}\ .$$
\end{proof}

%
%
%
%
%
%

%

%

\section{Real line bundles on a torus}\label{RLBT}

\subsection{Abelian Yang-Mills theory on a torus}

Let $T=V/\Lambda$ be a $n$-dimensional torus, where $V$ is an $n$-dimensional real vector space of dimension $n$ and $\Lambda\subset V$ a rank $n$ lattice such that $\langle\Lambda\rangle_\R=V$.   Let $u\in \mathrm{Alt^2}(\Lambda,\Z)= H^2(T,\Z)$ be an alternating $\Z$-valued form on $\Lambda$; we will denote by the same symbol  the corresponding differential form $u\in A^2(V)$ on $V$, and by $\bar u$ the  differential form on $T$ whose pull-back via the projection $p:V\to T$ is $u$; $\bar u$ is  the harmonic representative of the 2-cohomology class    $u\in H^2(T,\Z)$ with respect to any flat metric  on $T$ induced by an inner product on $V$.

Let $L$ be a Hermitian line bundle of Chern class $u$ on $T$, and ${\cal T}(L)$ the torus of Yang-Mills connections on $L$.

Our first goals are:
\begin{enumerate}
\item describe explicitly the torus ${\cal T}(L)$ of Yang-Mills connections on $L$,
\item for every Yang-Mills class $[A]$ describe the holonomy with respect to $A$ along the loops of the form $p[v_0,v_0+\lambda]\subset T$, $v_0\in V$, $\lambda\in\Lambda$.
\end{enumerate} 

Let $A$ be any Hermitian connection on $L$.  We define a map $\alpha^A:\Lambda\to S^1$  by the condition 
$$h^A_{c_\lambda}( \zeta)=\alpha^A_\lambda \zeta\ ,
$$
where $c_\lambda$ is the loop (based in the origin $0_T\in T$) defined by $c_\lambda(t):=p(\lambda t)$, and  $h^A$ stands for the holonomy associated with the connection $A$. The loops $c_{\lambda'}*c_\lambda$ and $c_{\lambda+\lambda'}$ are homotopic. 

We will show below that  $\alpha^A$ defines a semicharacter. In order to see this we need a general version of the holonomy formula which we recall now briefly:\vspace{4mm}

Let $Y$ be a differentiable manifold,  $L$ a Hermitian line bundle on $Y$ endowed with a Hermitian connection $A$, and let $c:[0,1]\to Y$ be a loop in $Y$ with $c(0)=c(1)=y_0$. Suppose that $c$ is homotopically trivial, i.e., there exists a smooth map 
$C:Q\to Y$, where $Q:=[0,1]\times  [0,1]$,  such that
$$\resto{C}{R}\equiv y_0\ , \ f(\cdot ,1)=c\ . 
$$
Here $R:=(\{0,1\}\times I)\cup (I\times\{0\})=\partial Q\setminus \left[I\times\{1\}\right]$. Then 
\begin{pr} The holonomy with respect to the connection $A$ along a loop $c$ can be computed using the formula
\begin{equation}\label{holof} h_c^A(\zeta)=e^{\int_Q C^*(F_A)} \zeta\ ,\ \forall \zeta\in L_{y_0}\ .
\end{equation}
\end{pr}
 \vspace{4mm}

Using the holonomy formula (\ref{holof}) and supposing that $\lambda$, $\lambda'$ are linearly independent over $\R$, one obtains
\begin{equation}\label{alphaid}
(\alpha^A_{\lambda+\lambda'})^{-1}\alpha^A_{\lambda'}\alpha^A_\lambda=e^{\int_{T(\lambda,\lambda')} p^*(F_A)}\ ,
\end{equation}
where $T(\lambda,\lambda')\subset V$ is the triangle given by the convex hull of the points $0_V$, $\lambda$, $\lambda'$ oriented in the obvious way.

Suppose now that $A$ is Yang-Mills with respect to the flat metric on $T$ induced by any inner product on $V$.  The Yang-Mills condition means that     $\frac{i}{2\pi} F_A$ is the harmonic representative  of the Chern class $c_1(L)\in H^2(T,\Z)$. We denote by $u$ the corresponding element in $\mathrm{Alt}^2(\Lambda,\Z)$ and we agree to use the same symbol for the corresponding constant differential form $u\in A^2(V)$ on $V$. With this notation, the Yang-Mills condition becomes
$$p^*(\frac{i}{2\pi} F_A)= u\ ,
$$
hence the holonomy identities (\ref{alphaid}) become
$$\alpha^A_{\lambda+\lambda'}=\alpha^A_\lambda\alpha^A_{\lambda'} e ^{2\pi i \int_{T(\lambda,\lambda')} u}=\alpha^A_\lambda\alpha^A_{\lambda'} e ^{\pi i u(\lambda,\lambda')}\ .
$$
\begin{dt} A map $\alpha:\Lambda\to S^1$ is called $u$-character if the following identity  holds:
$$\alpha_{\lambda+\lambda'}=\alpha_\lambda\alpha_{\lambda'} e^{\pi i u(\lambda,\lambda')}\ \hbox{ for }\lambda,\ \lambda'\in\Lambda.
$$
\end{dt} 

Note that (if non-empty) the set $\Hom_u(\Lambda,S^1)$ of $u$-characters is a $\Hom(\Lambda,S^1)$-torsor, so it has a natural differentiable structure which makes it  (non-canonically)  diffeomorphic to  the dual torus $V^\vee/\Lambda^\vee$.

\begin{pr} $\Hom_u(\Lambda,S^1)$ is non-empty and the assignment $A\mapsto \alpha^A$ defines a canonical isomorphism  
$$h:{\cal T}(L)\to \Hom_u(\Lambda,S^1)$$
 of differentiable $\Hom(\Lambda,S^1)$-torsors.   
\end{pr}
\pf Since for any Yang-Mills connection $A$ the corresponding map $\alpha^A$ is a $u$-character,  the first statement is clear. Note  that the tensor product of connections defines   a natural ${\cal T}_0$-torsor structure on ${\cal T}(L)$, where ${\cal T}_0$ stands for the torus of flat Yang-Mills connections on $T$. But ${\cal T}_0$ can be identified with $\Hom(\pi_1(T,0_T),S^1)=\Hom(\Lambda,S^1)$ via the holonomy map $h$ by a classical result in differential geometry. It suffices to notice that the tensor product of abelian connections corresponds to the multiplication of holonomies, so $h$ defines a morphism of $\Hom(\Lambda,S^1)$-torsors.
\qed

\begin{re} \label{extension} One can prove directly that the set of $u$-characters is non-empty:

Choose a basis $(e_1,\dots,e_n)$ of $\Lambda$. Then any system $z=(z_i)_{1\leq i\leq n}$  of elements in $S^1$ defines a  unique $u$-character $\alpha$ with the property  $\alpha(e_i)=z_i$.
 
 \end{re}

Our next goal is to describe explicitly the Yang-Mills connection which corresponds to a given $u$-character $\alpha$, and in particular to compute the holonomy associated with this connection along more general loops.

Let $(e_1,\dots,e_n)$ be a $\Z$-basis of $\Lambda$ (so also a $\R$-basis of $V$), and let $(x^1,\dots, x^n)$ be the dual basis of $V^\vee$. Any $x_i$ will be regarded as a smooth function on $V$.

Putting $u_{jk}:= u(e_j,e_k)$ one has
$$u=\frac{1}{2} \sum_{j,k} u_{jk} dx^j\wedge dx^k=\sum_{j<k}  u_{jk} dx^j\wedge dx^k\ .
$$
Let $\mathrm{v}$ be the tangent field on $V$ given by 
$$\mathrm{v}_v=v\in T_v(V)=V\ \ \forall v\in V\ .$$
We define the imaginary differential 1-form $\theta_u$ on $V$ by
$$\theta_u:=-\pi i\iota_\mathrm{v} u=-\pi i\sum_{j,k} u_{jk} x^jdx^k\ .
$$
One has $d\theta_u=-2\pi iu$, which shows that $\theta_u$ is the connection form of a connection $A_u$ on the trivial Hermitian line bundle $V\times\C$ with curvature $F_{A_u}=-2\pi i u$ and Chern form $c_1(A_u)=\frac{i}{2\pi} F_{A_u}=u$. The covariant derivative corresponding to $A_u$ is given by the formula $\nabla_u:= d+ \theta_u$.

In order to  compute the holonomy of $A_u$ along a segment $[v_0, v_0+w]\subset V$, we parameterize this segment by $c_{v_0,w}:[0,1]\to V$, $c_{v_0,w}(t):=v_0+tw$. The covariant derivative of the pull-back connection $c^*(A_u)$ on the trivial line bundle $[0,1]\times\C$ over $[0,1]$ is:   
$$d+c^*(\theta_u)=d-\pi i\sum_{j,k} u_{jk}(v_0+tw)^j w^k dt=d-\pi i\sum_{j,k} u_{jk}v_0^j w^k dt=d-\pi i  u(v_0,w)\ .
$$
Therefore the parallel transport in  the trivial line bundle $V\times\C$  along $c$ with respect to $A_u$ and with the initial condition $\zeta(0)=1$ is defined by the Cauchy problem:
$$\dot \zeta- \pi i u(v_0,w)\zeta=0\ ,\ \zeta(0)=1\ .
$$
This has the obvious solution $\zeta(t)= e^{\pi i u(v_0,w)t}$.   Therefore
\begin{re}\label{holo}
The holonomy 
$$h_{c_{v_0,w}}^{A_u}:\{v_0\}\times\C\to \{v_0+w\}\times \C$$
of the connection $A_u$ along $c_{v_0,w}$ is given by $h_{c_{v_0,w}}^{A_u}(\zeta)= e^{\pi i u(v_0,w)} \zeta$.
\end{re}

The action of the lattice $\Lambda$ on $V$ can be lifted to an action on the trivial line bundle $V\times \C$ by choosing a   factor of automorphy  $e=(e_\lambda)_{\lambda\in\Lambda}$, i.e., a system of functions $e_\lambda:V\to S^1$ satisfying the identities:
\begin{equation}\label{autid}
e_{\lambda'} (v+\lambda) e_\lambda(v)=e_{\lambda+\lambda'}(v)\ \  \forall\ \lambda,\ \lambda'\in\Lambda\ .
\end{equation}
The $\Lambda$-action $\tilde e$ on $V\times\C$ corresponding to a  factor of automorphy $e$ is given by 
$$\lambda\cdot (v,z)=\tilde e_\lambda(v,z):=(v+\lambda,e_\lambda(v)z)\ .
$$

Denote by $L(e)$ the line bundle on $T$ obtained as the $\Lambda$-quotient of the trivial line bundle $V\times\C$ with respect to $\tilde e$. We seek factors of automorphy  $e$ such that the connection $A_u$ descends to a connection $A(e)$ on the quotient line bundle $L(e)$. This is equivalent to the  condition   that $A_u$ is $\tilde e$-invariant.\\

Let $s_0$ be the constant section $s_0(v):=(v,1)$ in $V\times\C$, and denote by $T_\lambda:V\to V$ the translation defined by $\lambda$. The condition $\tilde e_\lambda^*(A_u)=A_u$ is equivalent with:
$$\nabla_u((\tilde e_\lambda)_*(s_0))=(\tilde e_\lambda)_*(\nabla_u(s_0))\ .
$$
Note that $(\tilde e_\lambda)_*(s_0)=(e_\lambda\circ T_{-\lambda})s_0$, hence the invariance condition becomes:
$$d e_\lambda+ e_\lambda T_\lambda^*\theta_u=e_\lambda \theta_u\ ,
$$
or, equivalently,
$$d\log(e_\lambda)=\pi i\ \iota_{\underline{\lambda}} u=d( u(\lambda,\cdot))\ .
$$
Therefore, the factor of automorphy $e$ must have the form
$$e_\lambda(v)=a_\lambda  e^{\pi i u(\lambda,v)}\ ,
$$
where $a_\lambda$ is a constant. Since $e$ must satisfy the cocycle  condition (\ref{autid}) we see that the function $a:\Lambda\to S^1$ has to satisfy the condition 
$$a_{\lambda+\lambda'}=a_\lambda a_{\lambda'} e^{\pi iu(\lambda,\lambda')}\ ,
$$
i.e., $a$ is a $u$-character.   Note that, by Remark \ref{holo}, the holonomy of the connection $A_u$  along the segment $[0,\lambda]\subset V$ is trivial.  Taking into account that, in the construction of $L(e)$, the identification between the fibers $\{0\}\times\C$, $\{\lambda\}\times\C$ is defined by $e_\lambda(0)$, one sees that the $u$-character $\alpha^{A(e)}$ associated to the Yang-Mills connection $A(e)$ is given by $\alpha^{A(e)}=\bar a$. This gives a geometric interpretation of the factor $a_\lambda$ appearing in the expression of $e_\lambda$.

\subsection{Real Yang-Mills connections on a torus}\label{RYMT}

Let $\tau:\Lambda\to\Lambda$ be an automorphism of order 2 of the lattice $\Lambda$; denote by the same symbol the induced involutions $V\to V$ and $T\to T$. Recall (see Corollary \ref{RealLB} )
%
%
that the group $H^1_{\Z_2}(T,\underline{S}^1(1))$ of  isomorphism classes of Real Hermitian line bundles on $T$ fits into the short exact sequence
\begin{equation}\label{shes}
0\to\qmod{[\Lambda^\vee]^{\tau^*}}{(\id+\tau^*)\Lambda^\vee}\to H^1_{\Z_2}(T,\underline{S}^1(1))\textmap{c_1}  \wedge^2\Lambda^\vee(1)^{\Z_2}\textmap{{\scriptscriptstyle{\mathcal O}}} \qmod{[\Lambda^\vee]^{-\tau^*}}{(\id-\tau^*)\Lambda^\vee}\ . 
\end{equation}
The obstruction space $\qmod{[\Lambda^\vee]^{-\tau^*}}{(\id-\tau^*)\Lambda^\vee}$ coincides with the cohomology group 
$$H^1_{\Z_2}( \Lambda^\vee)=H^2_{\Z_2}(\Hom(\Lambda, S^1))=\qmod{\Hom(\Lambda,S^1)^\tau}{(\id\cdot \tau)\Hom(\Lambda,S^1) }\ .$$

We know that a Hermitian  line bundle  $L$ admits $\tau$-Real structures if and only if the fixed point locus of the induced involution $\tau^*:{\cal T}(L)\to {\cal T}(L)$ is non-empty. Using the identification between ${\cal T}(L)$ and the space $\Hom_u(\Lambda,S^1)$ of $u$-characters the involution $\tau^*$ becomes:
$$\tau^*(\alpha)_{\lambda}=\bar\alpha_{\tau(\lambda)}\ .
$$

The existence of a $\tau$-Real structure on $L$ is therefore equivalent to the existence of a $u$-character $\alpha$ satisfying the $\tau$-Reality condition
\begin{equation}\label{tauReal} \alpha_{\tau(\lambda)}=\bar\alpha_\lambda\ .
\end{equation}

Using this remark one can compute explicitly the obstruction  ${\scriptstyle{\mathcal O}}(u)$  of a form $u\in\wedge^2\Lambda^\vee(1)^{\Z_2}$ in the following way: Fix any $u$-character $\alpha$, and consider  the function $\rho_\alpha:\Lambda\to S^1$ given by $\lambda\mapsto \alpha(\lambda)\alpha(\tau(\lambda))$; since $\tau^*(u)=-u$, this function is a $\tau$-invariant character, and its class  modulo $(\id\cdot \tau)\Hom(\Lambda,S^1)$ is independent of $\alpha$. This class is the obstruction ${\scriptstyle{\mathcal O}}(u)$.

\begin{pr}  \label{exRealstr} On a torus the obstruction map ${\scriptstyle{\mathcal O}}$  vanishes.
\end{pr}
\pf Choose a Comessatti basis $(\alpha_1,\dots,\alpha_a,\beta_1,\dots\beta_s,\gamma_{s+1},\dots,\gamma_{n-a})$ of  $(\Lambda,\tau)$, i.e., a basis satisfying 
\begin{equation}\label{ComBasis}
\tau(\alpha_i)=\alpha_i \hbox{ for } 1\leq i\leq a\ ;\ \tau(\beta_j)=\alpha_j-\beta_j  \hbox{ for } 1\leq j\leq s\ ; $$
$$ \tau(\gamma_k)=-\gamma_k\   \hbox{ for } s+1\leq k\leq n-a\ .
\end{equation}
Here  $s\leq a$ is the Comessatti characteristic  of $(\Lambda,\tau)$, and $n$ is the rank of $\Lambda$ (see Lemma 3.5 in \cite{S}).
As in Remark \ref{extension} we see that for any system 
$$z=(u_1,\dots, u_a,v_1,\dots v_s,w_{s+1},\dots,w_{n-a})$$
 of elements in $S^1$  we get a $u$-character $\alpha_z$ such that $ \alpha_z(\alpha_i)=u_i$, $\alpha_z(\beta_j)=v_j$, $\alpha_z(\gamma_k)=w_k$.  Note that the $\tau$-Reality  condition (\ref{tauReal}) holds if and only if it holds for the elements  of the Comessatti basis (which is $\tau$-invariant). Therefore $\alpha_z$ is $\tau$-Real if and only if 
$$u_i\in\{\pm 1\} \hbox { for } 1\leq i\leq a\ \hbox{ and }\ u_j=e^{\pi i u(\alpha_j,\beta_j)} \hbox{ for } 1\leq j\leq s\ .
$$
\qed

\begin{re} The proof of Proposition \ref{exRealstr} shows  that, in the presence of a Comessatti basis,  the space of $\tau$-Real $u$-characters  can be explicitly  identified with the space $\{\pm 1\}^{a-s}\times [S^1]^{n-a}$. The connected components of this space correspond bijectively to isomorphism classes of $\tau$-Real line bundles.
\end{re}
\vspace{2mm}

Our goal   now is the following: Let $(L,\tilde\tau)$ be a $\tau$-Real Hermitian line bundle on $T$. The fixed point locus $L^{\tilde \tau}$ is a $\R$-line bundle on the fixed point locus $T^\tau$, and we want to compute its Stiefel-Whitney class $w_1(L^{\tilde \tau})$. The fixed point locus $T^\tau$ is a disjoint union of components $T_{[\mu]}$, all translations of the torus $V^\tau/\Lambda^\tau$ by elements $[\mu]\in \frac{1}{2}\Lambda^{-\tau}/\frac{1}{2} (\id-\tau)\Lambda$, so the Stiefel-Whitney class $w_1(L^{\tilde \tau})$ can be regarded as an function
$$w:\qmod{\frac{1}{2}\Lambda^{-\tau}}{ \frac{1}{2} (\id-\tau)\Lambda}\map \Hom(\Lambda^\tau,\{\pm 1\})\ .
$$

For a class $[\mu]$ the morphism $w([\mu])\in \Hom(\Lambda^\tau,\{\pm 1\})$ has a simple geometric interpretation: $w([\mu])(\lambda)$  is just the holonomy of any $\tilde\tau$-invariant Hermitian connection $A$ on $L$ along the path $p\circ c_{\mu,\lambda}$.  This follows from the following general 
\begin{re} \label{HolSW} Let $F\to B$ be an Euclidean line bundle on a differentiable manifold $B$, and $A$ its unique $\mathrm{O}(1)$-connection (which is automatically flat). Let $\gamma:S^1\to B$ be a smooth map, and $h=\gamma_*([S^1])\in H_1(B,\Z)$. Then
$$\langle w_1(F),[h]_2\rangle=h_\gamma ^A\in\{\pm 1\}\simeq\Z_2\ .
$$
\end{re}

Suppose now that $L=L(e)$  where $e$ is the factor of automorphy associated with a $\tau$-Real $u$-character $a$. In this case one can use the Yang-Mills connection $A(e)$,  which is  $\tau$-Real. Using Remark \ref{holo} we see that the holonomy along the closed path $p\circ c_{\mu,\lambda}$ is given by
$$h(\zeta)= e_\lambda(\mu)^{-1}e^{\pi i u(\mu,\lambda)} \zeta=\bar a_\lambda e^{-\pi  i u(\lambda,\mu)+\pi i u(\mu,\lambda)}\zeta=\alpha_\lambda e^{\pi i u(2\mu,\lambda)}\zeta\ .
$$
Therefore we get:
$$w([\mu])(\lambda)=\alpha_\lambda e^{\pi i u(2\mu,\lambda)}\ \ \forall \lambda\in\Lambda^\tau\ .
$$

Note that $w(0)(\lambda)=\alpha_\lambda$, so that one finally obtains the transformation  formula
$$w([\mu])(\lambda)=w(0)(\lambda) e^{\pi i u(2\mu,\lambda)}\ \ \forall \lambda\in\Lambda^\tau\ .
$$
Identifying $H_1(T_{[\mu]},\Z)$ with $\Lambda^\tau$ and $H^1(T_{[\mu]},\Z_2)$ with $\Hom(\Lambda^\tau,\Z_2)$ we get the following {\it difference formula}
\begin{pr}  \label{difference} One has
\begin{equation} w([\mu])-w(0)=u(2\mu,\cdot) \hbox { (mod 2) }  .
\end{equation}
\end{pr}

This formula shows that the function $w$ is completely determined by $w(0)$. Note also that the difference  $w([\mu])-w(0)$ vanishes on the subgroup $(\id+\tau)\Lambda\subset \Lambda^\tau$ of  trivial invariants  in $\Lambda$. Indeed, one has
$$u(2\mu, \lambda+\tau(\lambda))=u(2\mu, \lambda)+u(2\mu,\tau(\lambda))=u(2\mu, \lambda)-u(2\tau(\mu),\lambda)=2u(2\mu,\lambda)\in 2\Z\ .
$$
It follows that the restriction $\resto{w([\mu])}{(\id+\tau)\Lambda}$ is independent of $[\mu]$. We want to identify this restriction.  \\

Note first that the map $\tilde f_u:\Lambda\to \Z_2$ defined by  $\tilde f_u(\lambda):=u(\lambda,\tau(\lambda))$ (mod 2) is a group morphism. This morphism vanishes on the subgroup $\Lambda^{-\tau}$ of anti-invariant elements, because for any $\nu\in \Lambda^{-\tau}$ one has
$$u(\lambda+\nu,\tau(\lambda+\nu))=u(\lambda,\tau(\lambda))+u(\lambda,\tau(\nu))+u(\nu,\tau(\lambda))+u(\nu,\tau(\nu))=$$ $$=u(\lambda,\tau(\lambda))+2u(\lambda,\tau(\nu))\ .
$$
Therefore the morphism $\tilde f_u$ induces a well-defined morphism $f_u:(\id+\tau)\Lambda\to \Z_2$ given by 
$$f_u(\lambda+\tau(\lambda)):= \tilde f_u(\lambda)=u(\lambda,\tau(\lambda))\ (mod\ 2).$$
  Note also that the morphism $\mathrm{Alt}^2(\Lambda,\Z)\to \Hom((\id+\tau)\Lambda, \Z_2)$ given by $u\mapsto f_u$ is obviously a group morphism. Using the identification $\{\pm 1\}=\Z_2$ we have
\begin{re} \label{restriction} For any $[\mu]\in  {\frac{1}{2}\Lambda^{-\tau}}/{ \frac{1}{2} (1-\tau)\Lambda}$ one has $\resto{w([\mu])}{(\id+\tau)\Lambda}=f_u$.
\end{re}
\pf  The holonomy along the closed path $p\circ c_{\mu,\lambda+\tau(\lambda)}$ is given by
$$h(\zeta)=\alpha_{\lambda+\tau(\lambda)} =e^{\pi i u(2\mu,\lambda+\tau(\lambda))}\zeta \ .
$$
 Since $ u(2\mu,\lambda+\tau(\lambda))=u(2\mu,\lambda)+u(2\mu,\tau(\lambda))=u(2\mu,\lambda)-u(2\tau(\mu),\lambda)=2 u(2\mu,\lambda)\in \Z$, one gets
$$w([\mu])(\lambda+\tau(\lambda))=\alpha_{\lambda}\bar\alpha_{\lambda} e^{\pi  i u(\lambda,\tau(\lambda))}=e^{\pi  i u(\lambda,\tau(\lambda))}\ .
$$

\qed

We can now summarize our results and give an explicit formula for the Stiefel-Whitney class $w_1(L^{\tilde\iota})$ of the fixed point bundle $L^{\tilde\iota}$ of a $\tau$-Real line bundle on a torus $T$. We may suppose that $L$ is endowed with an invariant Yang-Mills  connection $A=A(e)$, where $e$ is the factor of automorphy  associated with a $\tau$-Real $u$-character $a=(a_\lambda)_{\lambda\in\Lambda}$.

\begin{pr} \label{GiveSW} For   elements $\lambda\in\Lambda^\tau$ and $ [\mu]\in  {\frac{1}{2}\Lambda^{-\tau}} /{ \frac{1}{2} (\id-\tau)\Lambda}$ 
one has
$$w([\mu])(\lambda)=\bar a_\lambda e^{\pi i u(2\mu,\lambda)}\ .
$$
In particular, $w(0)(\lambda)=\bar a_\lambda$ for every $\lambda\in \Lambda^\tau$.
\end{pr}

\subsection{Classification of Real line bundles on a Real torus} \label{ClassRLBTorus}

Let again $\tau:\Lambda\to\Lambda$ be an automorphism of order 2 of an $n$-dimensional lattice $\Lambda\subset V=\langle \Lambda\rangle_\R$, and denote by the same symbol the induced involutions on $V$ and $T=V/\Lambda$.

Our next goal is a complete description --  in  terms of characteristic classes -- of the  Grothendieck group $H^1_{\Z_2}(T,\underline{S}^1(1))$ of  $\tau$-Real line bundles on $T$. 

For any $u\in\mathrm{Alt}^2(\Lambda,\Z)^{-\tau^*}$ we put
$$W(u):=\{ w\in \Hom(\Lambda^\tau,\Z_2)|\ \resto{w}{(\id+\tau)\Lambda}=f_u\}\ .
$$
Consider the fiber product 
$$\mathrm{Alt}^2(\Lambda,\Z)^{-\tau^*}\times_{\Hom((\id+\tau)\Lambda,\Z_2)} \Hom(\Lambda^\tau,\Z_2)\ ,$$
where $\mathrm{Alt}^2(\Lambda,\Z)^{-\tau^*}$, $\Hom(\Lambda^\tau,\Z_2)$ are regarded as   groups over  ${\Hom((\id+\tau)\Lambda,\Z_2)}$ via $u\mapsto f_u$, and  $w\mapsto \resto{w}{(\id+\tau)\Lambda}$ respectively.

By Remark \ref{restriction} it follows that, for every $\tau$-Real line bundle $(L,\tilde\tau)$, the Stiefel-Whitney class of the restriction  $\resto{L^{\tilde\tau}}{T_0}$ of the real line bundle $L^{\tilde\tau}$ to the standard connected component  $T_0:=V^\tau/\Lambda^\tau$ of $T^\tau$ is an element of $W(u)$.
\begin{thry} \label{ClassifRLBTorus} The group morphism 
$$cw_0:H^1_{\Z_2}(T,\underline{S}^1(1))\to \mathrm{Alt}^2(\Lambda,\Z)^{-\tau^*}\times_{\Hom((\id+\tau)\Lambda,\Z_2)} \Hom(\Lambda^\tau,\Z_2)$$ defined by
$$cw_0([L,\tilde\tau]):=(c_1(L),w_1(\resto{L^{\tilde\tau}}{T_0}))\ ,$$
 is a bijection.
\end{thry}
\pf \vspace{2mm}\\
1. {\it  Injectivity:}  An element of $\ker(cw_0)$ is the class of a $\tau$-Real line bundle $(L,\tilde\tau)$ with trivial first Chern class $c_1(L)$ and vanishing Stiefel-Whitney class $w_1(\resto{L^{\tilde\tau}}{T_0})$. This first condition implies that $(L,\tilde\tau)$ is induced by an element 
$$[\chi]\in \qmod{[\Lambda^\vee]^{\tau^*}}{(\id+\tau^*)\Lambda^\vee}\ ,$$
i.e.,  it coincides with  the flat  $\tau$-Real line bundle $(L_\chi,\tilde\tau_\chi)$, where $L_\chi$ is defined by  the $\tau$-invariant representation $e^{\pi i  \chi}:\pi_1(T)\to\{\pm 1\}$ associated with    $\chi$; $L_\chi$ can be endowed with a natural $\tau$-Real structure $\tilde\tau_\chi$.

Note   that the natural morphism
$$\qmod{[\Lambda^\vee]^{\tau^*}}{(\id+\tau^*)\Lambda^\vee}\to \Hom(\Lambda^\tau,\Z_2)
$$
is a monomorphism, and its image is 
$$ \Hom\left(\qmod{\Lambda^\tau}{(\id+\tau)\Lambda},\Z_2\right)\subset  \Hom(\Lambda^\tau,\Z_2)\ .$$
This can easily be proved  using a Comessatti basis in $\Lambda$. Therefore  the vanishing of $w_1(\resto{L^{\tilde\tau}}{T_0})$ implies $[\chi]=0$.
\\ \\
2. {\it Surjectivity:} Let $(u,w)\in
 \mathrm{Alt}^2(\Lambda,\Z)^{-\tau^*}\times_{\Hom((\id+\tau)\Lambda,\Z_2)} \Hom(\Lambda^\tau,\Z_2)$. Using the vanishing of the obstruction map ${\scriptstyle{\mathcal O}}$ (see Proposition \ref{exRealstr}) it follows that there exists a $\tau$-Real line bundle $(L',\tilde\tau')$ with $c_1(L')=u$.  Put $w':= w_1(\resto{{ L'}^{\tau'}}{T^0})$. We know by Remark \ref{restriction}  that $w'\in W(u)$. Therefore the difference $w-w'$ vanishes on $(\id+\tau)\Lambda\subset \Lambda^\tau$, so it defines a morphism $v\in \Hom(\Lambda^\tau/(\id+\tau)\Lambda,\Z_2)$. Let $[\chi]$  be the corresponding element in $ {[\Lambda^\vee]^{\tau^*}}/{(\id+\tau^*)\Lambda^\vee}$. Then  $(L,\tilde\tau):=(L',\tilde\tau')\otimes (L_\chi,\tilde\tau_\chi)$ is  a $\tau$-Real line bundle with $cw_0([L,\tilde\tau])=(u,w)$.
\qed

\section{Real theta line bundles}

\subsection{Holomorphic line bundles on a complex torus}

We will see that, using the Kobayashi-Hitchin correspondence  between abelian Hermite-Einstein connections and holomorphic line bundles, one can recover the classical Appell-Humbert theorem in a completely natural way.

Suppose that $J$ is a complex structure on $V$. Endow the torus $T = V/\Lambda$ with the induced holomorphic structure, and let $L$ be a Hermitian line bundle on $T$ whose Chern class $c_1(L)$ corresponds to $u\in\mathrm{Alt}^2(\Lambda,\Z)$. A connection $A\in{\cal A}(L)$ is Hermite-Einstein with respect to the flat Kähler metric on $T$ (defined by any Hermitian structure on $V$) if and only if it is Yang-Mills and its curvature is of type (1,1). Therefore the group $\Pic(T)$ of isomorphism classes of holomorphic line bundles on $T$ can be identified with the union $\coprod_{c_1(L)\in \mathrm{NS}(T)}{\cal T}(L)$, where $\mathrm{NS}(T)\subset H^2(T,\Z)$ is the Neron-Severi group of $T$. $\mathrm{NS}(T)$ can be identified with the subgroup  of $\mathrm{Alt}^2(\Lambda,\Z)$ consisting of forms $u$ whose $\R$-linear extension is $J$-invariant. This means that the corresponding differential form on $V$ is of type (1,1). Our goal is to find a natural holomorphic factor of automorphy $\epsilon$ for the holomorphic line  bundle ${\cal L}(e)$ which corresponds to the Hermite-Einstein connection $A(e)$ on $L(e)$. The holomorphic structure of ${\cal L}(e)$ is defined by the semi-connection $\bar\partial_{A(e)}$. The pull-back of this semi-connection to the trivial line bundle $V\times\C$  is given by $\bar\partial_u:=\bar\partial +\theta_u^{0,1}$, so it does not coincide with the trivial semi-connection $\bar\partial$ (unless of course $u=0$).  We want to construct explicitly a holomorphic line bundle ${\cal L}'(e)$  on $T$ which is holomorphically isomorphic to ${\cal L}(e)$ and whose pull-back to $V$ is the standard trivial holomorphic line bundle $(V\times\C,\bar\partial)$. This line bundle will be defined by {\it a holomorphic factor of automorphy} $\epsilon=(\epsilon_\lambda)_{\lambda\in\Lambda}$.

In order to obtain  this holomorphic factor of automorphy,  the first step is to find a complex gauge transformation $g\in {\cal C}^\infty(V,\C^*)$ such that $g^*(\bar\partial)=\bar\partial_u$, i.e., we have to solve the equation:
$$g^{-1} dg=\theta_u^{0,1}\ .
$$
If $g$ is a solution of the this equation, the corresponding factor of automorphy will be
$$e'_\lambda(v):=g(v+\lambda) e_\lambda(v) g^{-1}(v)\ .
$$

Using complex coordinates $z^j$ on $V$ and writing $u$ as
$$u=\frac{i}{2} \sum_{j,k}\omega_{jk} dz^j\wedge d\bar z^k
$$
with $\bar \omega_{ij}=\omega_{ji}$, one obtains
$$\theta_u^{0,1}=\frac{\pi}{2}\sum_{j,k} \omega_{jk} z^j d\bar z^k\ .
$$
Note that $\theta_u^{0,1}=\frac{\pi}{2} \bar\partial(\sum_{j,k}\omega_{jk} z^j\bar z^k)$, so one can take
$$g(v)=e^{\frac{\pi}{2} \sum_{j,k}\omega_{jk} v^j\bar v^k}\ .
$$
Recall that the Hermitian form associated with $u$ ($\C$-linear in the second variable) is given by
$$H_u(v,w):=u(v,Jw)+iu(v,w)\ .
$$
One checks that  $H(v,w)=\sum_{j,k} \omega_{jk} w^j \bar v^k$. Therefore the most natural solution is $g(v)=e^{\frac{\pi}{2} H_u(v,v)}$. With this choice of the complex  gauge transformation $g$, the corresponding holomorphic factor of automorphy of ${\cal L}'(e)\simeq {\cal L}(e)$ is given by:
$$\epsilon_\lambda(v)=g(v+\lambda) g^{-1}(v) e_\lambda(v)= a_\lambda e^{\pi(H(\lambda,v)+\frac{1}{2} H(\lambda,\lambda))}\ .
$$
This is the {\it canonical factor of automorphy}  in the sense of Mumford's [Mu].  Note that $g$ defines an isomorphism  of  Hermitian holomorphic line bundles on $V$
$$(V\times\C,\bar\partial_{A(e)},h_0)\map (V\times\C,\bar\partial, |g|^{-2} h_0)\ ,
$$
where $h_0$ is the standard Hermitian metric on the trivial line bundle $V\times\C$.  This isomorphism maps the Chern connection on the left (which is $A(e)$)  to the Chern connection of the pair $(\bar\partial, |g|^{-2} h_0)$, which descends to the Chern connection of ${\cal L'}(e)$ endowed with the metric induced by $|g|^{-2} h_0$ (which is the unique Hermite-Einstein connection of the holomorphic line bundle ${\cal L'}(e)$). This proves the following important theorem,  which yields the factor of automorphy $e$ of a Yang-Mills connection which is gauge equivalent to the Hermitian-Einstein connection of a holomorphic line bundle ${\cal L}(\epsilon)$ defined by a canonical factor of automorphy $\epsilon$.  
\begin{thry} \label{AFHE} Let  $\epsilon=(\epsilon_\lambda)_{\lambda\in\Lambda}$ with
$$\epsilon_\lambda(v)=a_\lambda e^{\pi(H(\lambda,v)+\frac{1}{2} H(\lambda,\lambda))}
$$
 be a canonical factor of automorphy for a holomorphic line bundle ${\cal L}(\epsilon)$ on $T$, where $a:\Lambda\to S^1$ is an $S^1$-valued $\mathrm{Im}(H)$-character. Then the Yang-Mills connection $A(e)$ defined by the factor of automorphy $e_\lambda(v)=a_\lambda  e^{\pi i \mathrm{Im}(H)(\lambda,v)}$ is gauge-equivalent to the unique Hermite-Einstein connection on the holomorphic bundle ${\cal L}(\epsilon)$.

In particular   $\bar a_\lambda\in S^1$ is  the holonomy of this Hermite-Einstein connection   along the loop  $p\circ c_{\lambda}$.
\end{thry}

Note that this theorem allows one to read off the holonomy of the Hermitian-Einstein connection on the holomorphic bundle ${\cal L}(\epsilon)$ along segments of the form $p\circ c_\lambda$.
%
%

\subsection{Theta line bundles of   Klein surfaces}\label{ThetaLB}

Let $(C,\iota)$ be a Klein surface of genus $g$, and $\Theta\subset \Pic^{g-1}(C)$ the geometric theta divisor defined by
$$\Theta:=\{{\cal L}\in\Pic^{g-1}(C)|\ h^0({\cal L})>0\}\ .
$$
Denote by
$$S(C):=\{[\kappa]\in\Pic^{g-1}(C)|\ [\kappa^{\otimes 2}]=[\omega_C]\}$$
the  set of theta characteristics of $C$. This set is naturally a $\Pic^0(C)_2$-torsor.
 For any $[\kappa]\in S(C)$  we consider the divisor
 $$\Theta_\kappa:=\Theta-[\kappa]\subset \Pic^0(C)
 $$
 which will be called the theta divisor associated with $\kappa$. Note that
 $$(-1)^*\Theta_\kappa=\Theta_\kappa\ .
 $$
 Denote by $\hat\iota$ the involution $\hat\iota:\Pic(C)\to\Pic(C)$ given by 
 $$\hat\iota([{\cal L}]) =[\overline{\iota^*({\cal L})}] \ .
 $$
Note that $\hat\iota(\Pic^d(C))=\Pic^d(C)$ for any $d\in\Z$   and  that $\hat\iota$ leaves invariant the geometric theta divisor $\Theta$  (because   $H^0(C,{\cal L})$ and $H^0(C,\hat\iota({\cal L}))$ are naturally anti-isomorphic).

Clearly if a theta characteristic $[\kappa]\in S(C)$ is $\iota$-Real (i.e. $\hat\iota([\kappa])=[\kappa])$ then one has 
$$\hat\iota(\Theta_\kappa)=\Theta_\kappa\ .$$
The set of $\hat\iota$-Real theta characteristics is non-empty; this set has been studied in \cite{GH} and \cite{N}.
We obtain a holomorphic line bundle
$${\cal L}_\kappa:={\cal O}_{\Pic^0(C)} (\Theta_\kappa)
$$
on $\Pic^0(C)$, which is symmetric in the sense that
$$(-1)^* {\cal L}_\kappa={\cal L}_\kappa\ .
$$
Note also that ${\cal L}_\kappa$ is naturally a $\hat\iota$-Real line bundle on $\Pic^0(C)$ since it is associated with a Real divisor. The first Chern class of ${\cal L}_\kappa$ is the element 
$$u_C\in H^2(\Pic^0(C),\Z)=\mathrm{Alt}^2(H^1(C,\Z),\Z)$$
  defined by the cup form $u_C:H^1(C,\Z)\times H^1(C,\Z)\to\Z$.

Our next goals are:
\begin{enumerate}

\item determine explicitly the Appel-Humbert data of ${\cal L}_\kappa$,
\item determine explicitly the element of $H^1_{\Z_2}(\Pic^0(C),\underline{S}^1(1))$ defined by the $\hat\iota$-Real line bundle ${\cal L}_\kappa$ on 
$\Pic^0(C)$
\end{enumerate}

 Clearly the first component of the Appel-Humbert datum defining ${\cal L}_\kappa$ is the Hermitian form $H_{u_C}$ associated with the cup form $u_C$.
The second component  $\chi_\kappa$  of the datum is an $u_C$-character which  takes values in $\{\pm 1\}$ since ${\cal L}_\kappa$ is symmetric (see Corollary 2.3.7 in \cite{BL}).
We obtain the identities
$$\chi_\kappa(\lambda+\lambda')=\chi_\kappa(\lambda)\chi_\kappa(\lambda') e^{\pi i u_C(\lambda,\lambda')}\ .
$$
Note that $\chi_\kappa$ is trivial  on $2 H^1(C,\Z)\subset H^1(C,\Z)$ and descends to a well-defined map
$$\bar\chi_\kappa:\qmod{H^1(C,\Z)}{2 H^1(C,\Z)}=H^1(C,\Z_2)\map \{\pm 1\}=\Z_2\ .
$$
This $\Z_2$-valued map satisfies the identity
$$\bar\chi_\kappa([\lambda]+[\lambda'])=\bar\chi_\kappa([\lambda])+\bar\chi_\kappa([\lambda'])+\overline{u_C(\lambda,\lambda')}\ .
$$
Hence $\overline{\chi}_\kappa$ is a   quadratic refinement of the (mod 2) cup form 
$$\bar u_C: H^1(C,\Z_2)\times H^1(C,\Z_2)\to\Z_2\ .$$

Let  $[\kappa]\in S(C)$. Mumford defines a map $q_\kappa:\Pic^0(C)_2\to\Z_2$ given by
$$[\eta]\mapsto h^0(\kappa\otimes\eta)-h^0(\kappa)\hbox{ (mod 2)}\ .
$$
With the canonical identification $\Pic^0(C)_2=H^1(C,\Z_2)$ and using Poincaré duality Mumford's theta form $q_\kappa$ becomes a map (denoted by the same symbol)  
$$q_\kappa:H_1(C,\Z_2)\to\Z_2$$
  which satisfies {\it the  Riemann-Mumford relations:}
\begin{equation}\label{RM} q_\kappa(\eta+\eta')=q_\kappa(\eta)+q_\kappa(\eta')+
\eta\cdot\eta'
\end{equation}
\begin{pr} \label{AlphaKappa} Suppose  ${\cal L}_\kappa$ is associated with the Appel-Humbert data $(H_{u_C},\chi_\kappa)$. Then
$$\chi_\kappa(\lambda)=(-1)^{q_\kappa(\overline{\lambda\cap[C]})}\ \ \forall \lambda\in H^1(C,\Z)\ .
$$
\end{pr}

\begin{proof} Using \cite{BL} Proposition 4.7.2. one obtains
$$\chi_\kappa(\lambda)=(-1)^{\mathrm{mult}_{[\frac{1}{2}\lambda]}(\Theta_\kappa)-
\mathrm{mult}_{[0]}(\Theta_\kappa)}\ \ \forall\lambda\in H^1(C,\Z)\ .
$$
Now we use the Riemann singularity theorem \cite{BL}: 
$$\mathrm{mult}_{[{\cal L}]}(\Theta)=h^0({\cal L})\ .
$$
One obviously has 
$$\mathrm{mult}_{[0]}(\Theta_\kappa))=\mathrm{mult}_{[\kappa]}(\Theta))=h^0(\kappa)\ ,
$$
$$\mathrm{mult}_{[\frac{1}{2}\lambda]}(\Theta_\kappa)=\mathrm{mult}_{\kappa\otimes\left[\frac{1}{2}\lambda\right]}(\Theta))=h^0\left(\kappa\otimes\left[\frac{1}{2}\lambda)\right]\right)\ .
$$
Therefore
$$\chi_\kappa(\lambda)=(-1)^{q_\kappa\left(\left[\frac{1}{2}\lambda\right]\right)} \ .
$$
Since the image of $\left[\frac{1}{2}\lambda\right]\in \Pic^0(C)_2$ in $H_1(C,\Z_2)$ via our identification is $\overline{\lambda\cap[C]})$ we get
 $$\chi_\kappa(\lambda)=(-1)^{q_\kappa(\overline{\lambda\cap[C]})}\  .
$$

\end{proof}

 Composing the map $\Fg_X$ defined in Proposition \ref{FC} with the morphism %
$$H^1_{\Z_2}(X,\underline{S}^1(1))\to H^1(X^\iota,\Z_2)$$
 which maps a $\iota$-Real line bundle $(L,\tilde\iota)$ to the first Stiefel-Whitney class  $w_1(L^{\tilde\iota})$ one obtains a morphism
 $$w:\Pic(X)^{\hat\iota}\to H^1(X^\iota,\Z_2)\ .
 $$
 %
 %
%
 
The following result --  in the case of {\it effective} Real theta characteristics -- was first obtained by Natanzon using the theory of real Fuchsian groups and their liftings \cite{N}.  We give a short direct proof in the general case, combining elementary differential geometric arguments with classical results of Atiyah, Johnson and Libgober.
 
\begin{thry} \label{main} Let $\kappa$ be a Real theta characteristic on $(C,\iota)$, and let $C_1,\dots,C_r$ be the connected components of the fixed point locus $C^\iota$. One has
$$q_\kappa([C_i]_2)=\langle w([\kappa]),[C_i]_2\rangle+1\ .
$$
\end{thry}

In order to prove this theorem we need some preparations:\\

Let $(C,g)$ be a closed, connected oriented Riemann surface.    We identify the $SO(2)$-frame bundle $P_g\to C$ of $(C,g)$ with  the sphere bundle $q:S(T_C)\to C$ in the natural way, and we fix a $\Spin$-structure $\sigma:Q\to S(T_C)$ of $C$.
A simple oriented closed curve $c$ in $C$ yields a simple closed curve $\tilde c\subset S(T_C)$ given by the unit tangent vectors of $c$  which are compatible with the orientation. This defines a homology class $[\tilde c]\in H_1(S(T_C),\Z_2)$. Changing the orientation of $c$ will give a different lift in $S(T_C)$, but the $\Z_2$-homology classes defined by the two lifts coincide. 
Any homology class $\eta\in H_1(C,\Z_2)$ can be represented by a union of pairwise disjoint simple closed curves $c_1,\dots,c_m$. 
We will need the following important {\it lifting result} of Johnson \cite{J}: \\\\ 
{\it  Let $z\in H_1(S(T_C),\Z_2)$ be the homology class of a fiber. Putting
$$\hat \eta:= \sum_{i=1}^m [\tilde c_i]+ m z
$$
defines a canonical lifting map  \^{ } $: H_1(C,\Z_2)\to H_1(S(T_C),\Z_2)$ 
 satisfying the identity}
$$\widehat{\eta+\eta'}=\hat \eta+\hat \eta'+(\eta\cdot \eta')z\ .
$$

Now let $\xi \in H^1(S(T_C),\Z_2)$ be the class of a $\Spin$-structure on $C$. This is equivalent to the condition $\langle \xi,z\rangle=1$.  Johnson defines a map  $\omega_\xi:H_1(C,\Z_2)\to \Z_2$ by
$$\omega_\xi(\eta):=\langle \xi,\hat \eta\rangle\ ,
$$
where the canonical lift $\hat\eta$ is defined by a  $\Spin$-structure  $\sigma:Q\to S(T_C)$ in the  class  $\xi$. This map satisfies the identity
$$\omega_\xi(\eta+\eta')=\omega_\xi(\eta)+\omega_\xi(\eta')+\eta\cdot \eta'\ .
$$
In other words $\omega_\xi$  is a quadratic refinement of the (mod 2) intersection  form.\\

Let $Q(H_1(C,\Z_2),\cdot)$  be the set of quadratic 
refinements of the 
(mod 2) intersection form and denote by $\Spin(C)$ the set of equivalence classes of $\Spin$-structures on $C$. Note that there is a well known bijection $\xi:S(C)\to\Spin(C)$ between the set of theta characteristics and the set of equivalence classes of $\Spin$-structures on $C$ \cite{A2}, \cite{N}.  We have just explained that Johnson's construction defines a map $\omega:\Spin(C)\to Q(H_1(C,\Z_2),\cdot)$. On the other hand, by the Riemann-Mumford relations, Mumford's construction yields a map $q:S(C)\to Q(H_1(C,\Z_2),\cdot)$.  Libgober  has shown \cite{L} that the following diagram commutes:
$$\begin{array}{ccc}
S(C)& & \\
&   {\searrow} \raisebox{1.5ex} {$q$} &  \\
\xi\downarrow&&Q(H_1(C,\Z_2),\cdot)\\
&\nearrow \hspace{-2mm}\raisebox{-1.5ex}{$\omega$}&\\
\Spin(C)
\end{array}
$$

More precisely one has
\begin{equation} \label{Li}
q_\kappa(\eta)=\omega_{\xi_{[\kappa]}}(\eta)\ \ \forall \eta\in H_1(C,\Z_2)\ . 
\end{equation}

Now let $g$ be a $\iota$-invariant Hermitian metric on $C$ and denote again by $S(T_C)$ the sphere bundle of the real  tangent bundle  $T_C$ of $C$.   Let  $\kappa$ be a holomorphic line bundle representing a   theta characteristic. We choose a holomorphic isomorphism $\phi:\kappa^{\otimes 2}\to \omega_C$ with the canonical line bundle $\omega_C$, and we endow $\kappa$ with the  Hermitian metric induced via $\phi$ from the real cotangent bundle $T^*_C$, which is the underlying differentiable line bundle of $ \omega_C$. Via the standard  identification $P_g=S(T_C)$ the $\Spin$-structure associated with $\kappa$ is the double cover 
$$\sigma:S(\kappa)\textmap{\otimes 2} S(\kappa^{\otimes2})\textmap{\phi}S(T^*_C)\textmap{\simeq} S(T_C)\ .$$
Note that, by the definition of the map $\xi$, $\sigma$ represents the class $\xi_{[\kappa]}$. The sphere bundle $S(\kappa)$ can be also regarded as  an $S^1$-bundle over $C$ via the composition $\rho:=q\circ \sigma$.  Note that the holomorphic  line bundle  $\omega_C$ comes with  a {\it canonical} anti-holomorphic $\iota$-Real structure $\tilde\iota_{\rm can}$ acting on local holomorphic 1-forms by $\eta\mapsto \iota^*(\bar\eta)$.  The induced involution on $S(T_C)$ is just the tangent map $\iota_*$ of $\iota$.
\begin{re} There exists an anti-holomorphic $\iota$-Real structure $\tilde\iota_{0}$ on $\kappa$ which lifts the canonical $\iota$-Real structure $\tilde\iota_{\rm can}$ on $\omega_C$ via $\phi\circ (\cdot)^{\otimes 2}$.  This  $\iota$-Real structure is unique up to sign.
\end{re} 
Indeed,  Let $\tilde\iota$ be any  anti-holomorphic $\iota$-Real structure on $\kappa$ (see  Proposition \ref{FC}).   The  $\iota$-Real structure induced on  $\omega_C$ via $\phi\circ (\cdot)^{\otimes 2}$ is well-defined and anti-holomorphic, so it is equivalent with $\tilde\iota_{\rm can}$ modulo a constant $\zeta\in S^1$. It suffices to put $\tilde\iota_0:=z\tilde\iota$, where $z$ is a square root of $\zeta$.
\qed
 \begin{lm} \label{mylemma}
 Let $\kappa$ be a theta characteristic endowed with a $\iota$-Real structure $\tilde\iota_0$ lifting $\tilde\iota_{\mathrm{can}}$, and let  $\sigma:S(\kappa)\to S(T_C)$ be  the associated $\Spin$-structure. Let $\gamma:S^1\to C^\iota$ be a parametrization   with unit speed  of a connected component $C_0\subset C^\iota$, and let $\Gamma_0\subset S(T_C)$ be the image of the tangent map $\gamma_*$. Replacing $\tilde\iota_0$ by $-\tilde\iota_0$ if necessary the following holds:
 \begin{enumerate}
 \item $\sigma(S(\resto{\kappa)^{\tilde \iota_0}}{C_0})=\Gamma_0$ and the obvious restrictions of $\sigma$, $\rho$ and $q$ define a commutative diagram
 $$\begin{array}{ccc}
 \resto{S(\kappa)^{\tilde \iota_0}}{C_0}&\textmap{=}&\resto{S(\kappa)^{\tilde \iota_0}}{C_0}\\  \\
\ \ \ \  \downarrow {\sigma_0}& &\downarrow \rho_0\\ \\
 \Gamma_0&\textmap{\resto{q}{\Gamma_0}}& C_0
 \end{array}
 $$
 \item The principal $\Z_2$-bundles   $\sigma_0:\resto{S(\kappa)^{\tilde \iota_0}}{C_0}\to \Gamma_0$, $\rho_0:\resto{S(\kappa)^{\tilde \iota_0}}{C_0}\to C_0$ are isomorphic via $\resto{q}{\Gamma_0}$.
 \end{enumerate}
\end{lm}
\begin{proof}

Note first that the restriction $\resto{S(T_C)^{\iota_*}}{C_0}$ of the fixed point locus $S(T_C)^{\iota_*}$ to $C_0$ is the disjoint union $\Gamma_0\cup -\Gamma_0$.    
\\ \\
{\it Claim:} Either $\sigma^{-1}(\Gamma_0)= \resto{S(\kappa)^{\tilde \iota_0}}{C_0}$,  or $\sigma^{-1}(-\Gamma_0)= \resto{S(\kappa)^{\tilde \iota_0}}{C_0}$. 
\\

Indeed, for any $l \in \resto{S(\kappa)^{\tilde \iota_0}}{C_0}$ one has $\sigma(l)\in  \resto{S(T_C)^{\iota_*}}{C_0}$, because $\tilde\iota_0$ is a lift of $\iota_*$ via $\sigma$.  Therefore   either $\sigma(l)\in \Gamma_0$, or  $\sigma(l)\in -\Gamma_0$. Two elements $l_1$, $l_2\in \resto{S(\kappa)^{\tilde \iota_0}}{C_0}$ have the same image via $\sigma$ if and only if $l_2=\pm l_1$, because $\sigma$ is equivariant  with respect to $S^1\textmap{(\cdot)^2} S^1$ and is  fiberwise equivalent to the standard double cover of $S^1$. Therefore the quotient $\qmod{ \resto{S(\kappa)^{\tilde \iota_0}}{C_0}}{\{\pm \id\}}$ is mapped injectively to $\resto{S(T_C)^{\iota_*}}{C_0}$. On the other hand, the three projections
$$\qmod{ \resto{S(\kappa)^{\tilde \iota_0}}{C_0}}{\{\pm \id\}}\to C_0\ ,\ \Gamma_0\to C_0\ ,\ -\Gamma_0\to C_0
$$
are all diffeomorphisms. Hence the image of $\qmod{ \resto{S(\kappa)^{\tilde \iota_0}}{C_0}}{\{\pm \id\}}$ via $\sigma$ must be a section  of $ \resto{S(T_C)^{\iota_*}}{C_0}\to C_0$, which proves the claim.\\

If we replace $\tilde \iota_0$ by $-\tilde\iota_0$ the new fixed point locus  $S(\kappa)^{-\tilde \iota_0}$ will be  $iS(\kappa)^{\tilde \iota_0}$, and the  multiplication  by $i$ on $S(\kappa)$ corresponds to the  multiplication  by $-1$  on $S(T_C)$. Therefore, replacing  $\tilde \iota_0$ by $-\tilde\iota_0$ if necessary, we can   assume that $\sigma^{-1}(\Gamma_0)= \resto{S(\kappa)^{\tilde \iota_0}}{C_0}$. This proves the first statement. The second follows directly from the first.
\end{proof}
\begin{co} \label{myco} Let $\xi_{[\kappa]}\in H^1(S(T_C),\Z_2)$ be the first Stiefel-Whitney class of the $\Z_2$-bundle $\sigma:S(\kappa)\to S(T_C)$. Then, for every connected component $C_0$ of $C^\iota$ one has
$$\langle \xi_{[\kappa]},[\Gamma_0]\rangle =\langle w([\kappa]),[C_0]_2\rangle\ .
$$
\end{co}
\vspace{6mm}
Now the proof of Theorem \ref{main} is immediate:

\begin{proof}  (of Theorem \ref{main}) Using Libgober's formula (\ref{Li}) and Corollary \ref{myco} one obtains:
$$q_\kappa([C_i]_2)=\omega_{\xi_{[\kappa]}}([C_i]_2)=\langle \xi_{[\kappa]}, \widetilde{[C_i]}_2+z\rangle=\langle \xi_{[\kappa]},\widetilde{[C_i]}_2\rangle+1=\langle w([\kappa]),[C_i]_2\rangle+ 1\ .
$$
\end{proof}

\begin{co} \label{comessatti} One has $q_\kappa([C^\iota]_2)\equiv s$ \hbox{ (mod 2)}, where $s$ is  the Comessatti characteristic of $(C,\iota)$.
\end{co}
\begin{proof} The Comessatti characteristic of $(C,\iota)$ is given by the formula $s=g+1-r$ (see \cite{S}).    On the other hand, using the results in Appendix B   it follows:
$$\langle w([\kappa]),[C^\iota]_2\rangle=\langle  c_1(\kappa),[C]\rangle\ (\hbox {mod }2)  $$
Applying Theorem \ref{main}  we get:
$$q_\kappa([C^\iota]_2)=\langle w([\kappa]),[C^\iota]_2\rangle+ r \ (\hbox{ mod }2)\ \equiv g-1+r \ ( \hbox{ mod }2)\ \equiv s\  ( \hbox{ mod }2) \ .
$$
\end{proof}
As we announced in the introduction, we like to summarize  some of Natanzon's results in our terminology. 
{\ }\\ \\
{\it Natanzon's results}:\\

Let $(C,\iota)$ be a Klein surface of type $(g,r,a)$ with $r>0$, and let $C^\iota=\coprod_{i=1}^r C_i$ be the decomposition of the fixed point set. The labeling of the following theorems corresponds to the labeling in \cite{N}.
\\ \\
{\bf Theorem 5.2.} (Natanzon 1995/1996) {\it Let $\kappa$ be an effective Real theta characteristic. Then one has:
$$\langle w([\kappa]),[C_i]_2\rangle+1=q_\kappa([C_i]_2)\ \ \forall i=1,\dots,r
$$
\\
}
{\bf Theorem 5.3.} (Natanzon 1995/1996) {\it Suppose $a(C,\iota)=1$. For every system $(w^1,\dots,w^r)$ of classes $w^i\in H^1(C_i,\Z_2)$ such that $\sum \langle w^i,[C_i]_2\rangle \equiv g-1$ (mod 2), there exists an effective Real theta characteristic $\kappa$ such that
$$\resto{w([\kappa])}{C_i}=w^i \ \ \forall i=1,\dots,r\ .
$$
}
\\
{\bf Theorem 5.4.} (Natanzon 1995/1996) {\it Suppose $a(C,\iota)=0$. For every system $(w^1,\dots,w^r)\ne (1,1,\dots,1)$ of classes $w^i\in H^1(C_i,\Z_2)$ such that $\sum \langle w^i,[C_i]_2\rangle \equiv g-1$ (mod 2), there exists an effective Real theta characteristic $\kappa$ such that
$$\resto{w([\kappa])}{C_i}=w^i \ \ \forall i=1,\dots,r\ .
$$
}
\\
{\bf Remark:}  When $a(C,\iota)=0$, and $\kappa$ is a Real theta characteristic with $\resto{w([\kappa])}{C_i}=1$ for all $i=1,\dots,r$, then $\kappa$ is not effective (see \cite{GH}, Corollary 5.3). 
\\

The proofs of Natanzon's results above are based on the following fundamental theorem:
\\ \\
{\bf Theorem 4.1/5.1.} (Natanzon 1990/1996)
{\it  There exists a bijection  between similarity classes of liftings of real Fuchsian groups defining $(C,\iota)$ and Real theta characteristics on $(C,\iota)$.
}
\vspace{1cm}\\
 Recall \cite{GH}  that the orientation obstruction $a(C,\iota)$ is 0 when $C/\langle\iota\rangle$ is orientable and $a(C,\iota)$ is 1 when not. One has $a(C,\iota)=0$ if and only if $C\setminus C^\iota$ has two connected components.

Note that the submodule $\left\langle\{[C_i]|\  1\leq i\leq r\}\right \rangle_\Z$ generated by the classes of the circles $C_i$  is obviously contained  in the $\iota_*$-invariant  submodule $H_1(C,\Z)^{\iota_*}$ of $H_1(C,\Z)$. We will see that this submodule together with the submodule  $(\id+\iota_*) H_1(C,\Z)$ of trivial invariants generates   $H_1(C,\Z)^{\iota_*}$.
We refer to \cite{CN} for the following  result:
\begin{lm}\label{bases} Let $(C,\iota)$  be a Klein surface of genus $g$, $r$ the number of connected components of $C^\iota$, and denote by $s =g-r+1$ its Comessatti characteristic. Choose an orientation of  $C^\iota$, introduce an order relation $(C_1,\dots,C_r)$ on its set of connected components, and  put $v_i:=[C_i]\in H_1(C,\Z)$.

\begin{enumerate}
\item Suppose that the quotient $\qmod{C}{\langle\iota\rangle}$ is orientable, and put $k:=\frac{s}{2}$ 
\begin{enumerate}
\item The homology group $H_1(C,\Z)$ admits a symplectic basis of the form
$$(v_1,\dots,v_{r-1},x_1,\dots,x_k,\iota_* x_1,\dots,\iota_* x_k,w_1,\dots w_{r-1},y_1,\dots, y_k,-\iota_* y_1,\dots,-\iota_* y_k)
$$
where $\iota_* w_i=-w_i$.
\item The associated basis  
$$(\underline{a},\underline{b}):=(v_1,\dots,v_{r-1},x_1+\iota_* x_1,\dots,x_k+\iota_* x_k,y_1+\iota_* y_1,\dots, y_k+\iota_* y_k,$$
$$ w_1,\dots, w_{r-1}, y_1,\dots, y_k,\iota_* x_1,\dots,\iota_* x_k)
$$
is symplectic,  its first $g$ basis vectors $(a_1,\dots,a_g)$ are $\iota_*$-invariant, whereas the last $g$ basis vectors $(b_1,\dots,b_g)$ satisfy:
$$\iota_* b_i=-b_i\hbox{ for } 1\leq i\leq r-1,\ \iota_* b_{j}=a_{k+j}-b_j \hbox{ for } r\leq j\leq r+k-1,
$$
$$
\iota_* b_l=a_{l-k}-b_l \hbox{ for } k+r\leq l\leq 2k+r-1=g $$
\end{enumerate}
\item Suppose that the quotient $\qmod{C}{\langle\iota\rangle}$ is not orientable.  There exists a symplectic basis 
$$(\underline{a},\underline{b})=(a_1,\dots,a_g,b_1,\dots,b_g)
$$
of $H_1(C,\Z)$ such that $a_i=v_i$ for $1\leq i\leq r$, the first $g$ basis vectors $(a_1,\dots,a_g)$ are $\iota_*$-invariant, whereas the last $g$ basis vectors $(b_1,\dots,b_g)$ satisfy the identities: 
$$\iota_* b_j=\left\{ \begin{array}{ccc}
-b_j-\sum_{i=1}^g a_i&\rm for& 1\leq j\leq r\\ \\
-b_j-a_j-\sum_{i=1}^g a_i &\rm for& r+1\leq j\leq g
\end{array}\right.
$$
\end{enumerate} 
\end{lm}
Using this lemma, the results of Costa and Natanzon \cite{CN}, and elementary arguments one obtains:
\begin{co} \label{InvGenByCircles} Let $(C,\iota)$ be a Klein surface with $C^\iota\ne\emptyset$.  
\begin{enumerate} 
\item   The natural map 
$$j:\left\langle\{v_i|\  1\leq i\leq r \}\right \rangle_\Z\oplus\big[ (\id+\iota_*) H_1(C,\Z)\big]\map H_1(C,\Z)^{\iota_*}$$
is always surjective.
\item If $a(C,\iota)=0$  then, 
orienting  the curves $C_i$ in a suitable way,  one has a short exact  sequence
$$1\map  \Z\sum_{i=1}^{r }v_i\map  \oplus_{i=1}^{r } \Z v_i\map \left\langle v_1,\dots,v_r  \right\rangle_\Z\to  0\ ,
$$
and  
$$\left\langle v_1,\dots,v_r\right \rangle_\Z\cap\big[ (\id+\iota_*) H_1(C,\Z)\big]=\left\langle 2v_1,\dots, 2 v_r \right \rangle_\Z\ .
$$
\item If $a(C,\iota)=1$, then the canonical epimorphism $$\oplus_{i=1}^{r } \Z v_i \map \left\langle v_1,\dots,v_r \right\rangle_\Z$$ is an isomorphism and    one has:
$$H_1(C,\Z)^{\iota_*}=\langle a_1,\dots,a_g\rangle\ ,\ (\id+\iota_*) H_1(C,\Z)=\langle 2a_1,\dots, 2a_g,a_{r+1},\dots,a_g,\sum_{i=1}^g a_i\rangle_\Z $$
$$=\langle 2a_1,\dots, 2a_g,a_{r+1},\dots,a_g,\sum_{i=1}^r a_i\rangle_\Z\ ,
$$
$$(\id+\iota_*) H_1(C,\Z) \cap \langle a_1,\dots,a_r\rangle =\langle 2a_1,\dots, 2a_r, \sum_{i=1}^r a_i\rangle_\Z=\langle 2a_1,\dots, 2a_{r-1}, \sum_{i=1}^r a_i\rangle_\Z\ .
$$
Moreover, writing the sum $\sum_{i=1}^{r} v_i$ in the form $\sum_{i=1}^{r}v_i=x+\iota_*(x)$,  one has $x\cdot\iota_*(x)\equiv s$ (mod  2), where $s$ is the Comessatti characteristic of the pair $(H_1(C,\Z),\iota_*)$.
\end{enumerate}

\end{co}

The  torus $\Pic^0(C)$ can be identified with the quotient $H^1(C,{\cal O})/H^1(C,\Z)$, where $H^1(X,\Z)$ is embedded in $H^1(X,{\cal O})$ via the composition
$$H^1(C,\Z)\textmap{\simeq} 2\pi i H^1(C,\Z)\hookrightarrow iH^1(C,\R) \hookrightarrow H^1(X,\C)\textmap{p^{0,1}} H^{0,1}(C)=H^1(C,{\cal O})\ ,
$$
and the Real structure $\hat\iota$ corresponds to  the Real structure defined by the involution $-\iota^*:H^1(C,\Z)\to H^1(C,\Z)$ on this quotient.
We can now conclude with the following theorem, which describes the image of the $\hat\iota$-Real  line bundle $({\cal L}_\kappa,\tilde{\hat\iota})$ on $\Pic^0(C)$ as element in the fiber product 
$$\mathrm{Alt}^2(H^1(C,\Z),\Z)^{(\iota^*)^*}\times_{\Hom((\id-\iota^*)H^1(C,\Z),\Z_2)} \Hom(H^1(C,\Z)^{-\iota^*},\Z_2)$$ appearing in our classification Theorem \ref{ClassifRLBTorus}.
\begin{thry} \label{LK} The element 
$$cw_0([{\cal L}_\kappa,\tilde{\hat\iota}])\in \mathrm{Alt}^2(H^1(C,\Z),\Z)^{(\iota^*)^*}\times_{\Hom((\id-\iota^*)H^1(C,\Z),\Z_2)} \Hom(H^1(C,\Z)^{-\iota^*},\Z_2)$$
is the pair $(u_C,w_\kappa)$, where $w_\kappa:H^1(C,\Z)^{-\iota^*}\to\Z_2$ is   defined as   the unique extension of $f_{u_C}:(\id-\iota^*)H^1(C,\Z)\to\Z_2$ which satisfies the equalities
\begin{equation}\label{ONCI}
w_\kappa([C_i]^\vee)=\langle w([\kappa]),[C_i]_2\rangle+1\ \hbox{ (mod 2)}\ .\end{equation}
\end{thry}
\begin{proof}
Recall that $w_\kappa:H^1(C,\Z)^{-\iota^*}\to\Z_2$ is given by the Stiefel-Whitney class of the restriction $\resto{{\cal L}_\kappa^{\tilde{\hat\iota}}}{T_0}$ of the real line bundle ${\cal L}_\kappa^{\tilde{\hat\iota}}$ to the standard connected component $T_0$ of $\Pic^0(C)^{\hat\iota}$. ${\cal L}_\kappa$ possesses a Hermitian-Einstein metric $h$, which  is $\tilde{\hat\iota}$-anti-unitary. Therefore the Hermitian-Einstein connection $A_\kappa$ on ${\cal L}_\kappa$ is compatible with ${\tilde{\hat\iota}}$, and hence the Stiefel-Whitney class of ${\cal L}_\kappa^{\tilde{\hat\iota}}$ is given by the holonomy of this connection along loops contained in $\Pic^0(C)^{\hat\iota}$ (see Remark \ref{HolSW}). On the other hand, by  Theorem \ref{AFHE} we can read off the factor of automorphy (and hence the holonomy along standard loops) of a Yang-Mills connection gauge equivalent to $A_\kappa$ from the canonical factor of automorphy of ${\cal L}_\kappa$. 

We apply now Proposition \ref{AlphaKappa} which computes this factor of automorphy in terms of Mumford's theta form $q_\kappa$ and Theorem \ref{main} which gives a geometric interpretation for $q_\kappa([C_i]^\vee)$. This proves that  $w_\kappa([C_i]^\vee)=\langle w([\kappa]),[C_i]\rangle+1$, as claimed. On the other hand we know, by the results in section \ref{ClassRLBTorus}, that $w_\kappa$ extends $f_{u_C}$. Finally, by Corollary \ref{InvGenByCircles}   the classes $[C_i]^\vee$ generate $H^1(C,\Z)^{-\iota^*}$ modulo the subgroup $(\id-\iota^*)H^1(C,\Z)$ of trivial anti-invariants, so (\ref{ONCI}) determines the extension $w_\kappa$.
\end{proof}

\begin{re}\begin{enumerate}
\item  The intersection $\langle [C_1]^\vee,\dots, [C_r]^\vee\rangle\cap (\id-\iota^*)H^1(C,\Z)$ is not trivial (see Corollary \ref{InvGenByCircles}). The map $f_{u_C}$ agrees with the map defined by the right hand side of  (\ref{ONCI}) on the intersection. This follows from our results, but can also be checked directly.
\item Using Theorem \ref{LK} and the  difference formula given by Proposition \ref{difference} we get the Stiefel-Whitney classes  of the restrictions of the real line bundle ${\cal L}_\kappa^{\tilde{\hat\iota}}$ to all connected components of $\Pic^0(C)^{\hat\iota}$.
\end{enumerate}
\end{re}

\subsection{Real determinant line bundles } \label{results}

Let $(C,\iota)$ be a Klein surface with $C^\iota\ne\emptyset$. We have seen that $\iota$ induces a Real structure (anti-holomorphic involution) $\hat\iota:\Pic(C)\to\Pic(C)$ on the Picard group of $C$ by
$$\hat\iota([{\cal L}])=[\iota^*(\bar{\cal L})]\ .
$$
This involution leaves the degree invariant, so it induces an anti-holomorphic involution on any connected component $\Pic^d(C)$.  The geometric theta divisor $\Theta\subset\Pic^{g-1}(C)$ defines a $\hat\iota$-invariant holomorphic line bundle $ {\cal O}_{\Pic^{g-1}(C)}(\Theta)$.

For every degree $d\in\Z$ we denote by $\hat{\hat\iota}$ the anti-holomorphic involution induced by $\hat\iota$ on $\Pic(\Pic^d(C))$. Note that $\hat{\hat\iota}$ maps $\Pic^c(\Pic^d(C))$ onto $\Pic^{-\hat\iota^*(c)}(\Pic^d(C))$ for every Chern class $c\in \NS(\Pic^d(C))$.

For every   $\lambda\in\Pic^{g-1}(C)$ we denote by $\Theta-\lambda\subset \Pic^0(C)$ the $(-\lambda)$-translate of the geometric theta divisor $\Theta$, and we denote by 
$${\cal L}_\lambda:={\cal O}_{\Pic^0(C)}(\Theta-\lambda)=(\tau_\lambda)^*{\cal O}_{\Pic^g(C)}(\Theta)$$
  the corresponding line bundle on $\Pic^0(C)$. The Chern class of ${\cal L}_\lambda$ is the element of $H^2(\Pic^0(C),\Z)=\mathrm{Alt}^2(H^1(C,\Z),\Z)$ defined by the cup form $u_C$ of $C$.
The assignment $\lambda\mapsto[{\cal L}_\lambda]$ defines a holomorphic map %
$$\varphi: \Pic^{g-1}(C)\to \Pic^{u_C}(\Pic^0(C))\ .$$
Note that $-\hat\iota^*(u_C)=u_C$, so the involution $\hat{\hat\iota}$ leaves $\Pic^{u_C}(\Pic^0(C))$ invariant.
\begin{lm} The map $\varphi:(\Pic^{g-1}(C),\hat\iota)\to (\Pic^{u_C}(\Pic^0(C)),\hat{\hat\iota})$ is an isomorphism of Real complex manifolds. 
\end{lm}
\begin{proof} For an element $\lambda_0\in\Pic^0(C)$  one has
$$\varphi(\lambda_0\otimes\lambda)= \tau_{\lambda_0\otimes\lambda}^*[{\cal O}_{\Pic^g(C)}(\Theta)]=\tau_{\lambda_0}^*\tau_\lambda^*[{\cal O}_{\Pic^g(C)}(\Theta)]= \tau_{\lambda_0}^*\varphi(\lambda)\ ,
$$
which shows that $\varphi$ commutes with the natural $\Pic^0(C)$-actions on   $\Pic^{g-1}(C)$ and $\Pic^{u_C}(\Pic^0(C))$. The first manifold is obviously a $\Pic^0(C)$-torsor, whereas the second is  a $\Pic^0(C)$-torsor because $u_C$ is a principal polarization of the torus $\Pic^0(C)$ (see \cite{BL}, p. 36-37). This proves that $\varphi$ is an isomorphism. On the other hand note that the Real structure $\hat{\hat\iota}$ induced on $\coprod_{d\in\Z}\Pic(\Pic^d(C))$ by $\hat\iota$ satisfies the identity
$$\hat{\hat\iota}(\tau_{\lambda}^*([\mathscr{L})])= \tau_{\hat\iota(\lambda)}^*(\hat{\hat\iota}([\mathscr{L}])\ .
$$
Since $\Theta$ is $\hat\iota$-invariant, it follows that the holomorphic line bundle ${\cal O}_{\Pic^{g-1}(C)}(\Theta)$ is $\hat{\hat\iota}$-invariant, so for any $\lambda\in \Pic^{g-1}(C)$ one has
$$\varphi(\hat\iota(\lambda))=\tau_{\hat\iota(\lambda)}^*[{\cal O}_{\Pic^g(C)}(\Theta)]=\tau_{\hat\iota(\lambda)}^*(\hat{\hat\iota}([{\cal O}_{\Pic^g(C)}(\Theta)])=\hat{\hat\iota}\left(\tau_{\lambda}^*[{\cal O}_{\Pic^{g-1}(C)}(\Theta)]\right)
$$
$$=\hat{\hat\iota}\left(\tau_{\lambda}^*[{\cal O}_{\Pic^{g-1}(C)}(\Theta)]\right)=\hat{\hat\iota}(\varphi(\lambda))\ ,
$$
which proves that $\varphi$ is Real.
\end{proof}

We are interested in the induced bijection
$$\pi_0(\varphi):\pi_0(\Pic^{g-1}(C)^{\hat\iota})\map \pi_0(\Pic^{u_C}(\Pic^0(C))^{\hat{\hat\iota}})\ .
$$

We know by Proposition \ref{FC} that $\pi_0(\Pic^{g-1}(C)^{\hat\iota})$ can be identified via the map $\mathcal{F}_C$ with the   subset of $H^1_{\Z_2}(C,\underline{S}^1(1))$ consisting of isomorphism classes of $\iota$-Real Hermitian line bundles $(L,\tilde\iota)$ with $\deg(L)=g-1$.
Using Theorem \ref{ClassOnKlein} we see that this subset can be identified with
$$ H^1(C^\iota,\Z_2)_{g-1}:=\{w\in H^1(C^\iota,\Z_2)|\ \deg_{\Z_2}(w)=g-1\hbox{ (mod 2)} \}\ .
$$

The composition of these identifications yields a bijection 
$$w_C^{g-1}:\pi_0(\Pic^{g-1}(C)^{\hat\iota})\to H^1(C^\iota,\Z_2)_{g-1}$$
 which can be explicitly described as follows: for a $\hat\iota$-invariant holomorphic line bundle ${\cal L}$ of degree $g-1$, we consider an anti-holomorphic $\iota$-Real structure $\tilde\iota$ on ${\cal L}$. Then $w_C^{g-1}$ maps the connected component of $[{\cal L}]$ in $\Pic^{g-1}(C)^{\hat\iota}$ to $w_1({\cal L}^{\tilde\iota})\in H^1(C^\iota,\Z_2)_{g-1}$.

Similarly,  the set $\pi_0(\Pic^{u_C}(\Pic^0(C))^{\hat{\hat\iota}})$ can be identified via the map $\mathcal{F}_{\Pic^0(C)}$ with the subset of the group $H^1_{\Z_2}(\Pic^0(C),\underline{S}^1(1))$ consisting of $\hat\iota$-Real Hermitian line bundles $({\cal M},\tilde{\hat\iota})$ on $\Pic^0(C)$ with $c_1({\cal M})=u_C$.  
Therefore, using the results in section \ref{ClassRLBTorus}, we see that the set  of $\hat\iota$-Real Hermitian line bundles $({\cal M},\tilde{\hat\iota})$ on $\Pic^0(C)=H^1(C,{\cal O})/H^1(C,\Z)$ with $c_1({\cal M})=u_C$ can be identified with
$$W(u_C)=\{w\in \Hom(H^1(C,\Z)^{-\iota^*},\Z_2)|\ \resto{w}{(\id-\iota^*)H^1(C,\Z)}=f_{u_C}\}\ .
$$
Note that the condition $\resto{w}{(\id-\iota^*)H^1(C,\Z)}=f_{u_C}$ simply means
\begin{equation}
w(\lambda-\iota^*(\lambda))=\langle \lambda,\iota^*(\lambda)\rangle\hbox{ (mod 2)}\ \ \forall\lambda\in H^1(C,\Z)\ .
\end{equation}

Composing these identifications we obtain a bijection
$$w_{\Pic^0(C)}^{u_C}:\pi_0(\Pic^{u_C}(\Pic^0(C))^{\hat{\hat\iota}})\to W({u_C})
$$
which can be explicitly described  as follows: for a $\hat{\hat\iota}$-invariant holomorphic line bundle ${\cal M}$ on $\Pic^0(C)$ with Chern class $u_C$ consider an anti-holomorphic $\hat\iota$-Real structure $\tilde{\hat\iota}$ on ${\cal M}$. Then $w_{\Pic^0(C)}^{u_C}$ maps the connected component of $[{\cal M}]$ in $ \Pic^{u_C}(\Pic^0(C))^{\hat{\hat\iota}}$ to $w_1(\resto{{\cal M}^{\tilde{\hat\iota}}}{T_0})$, where $T_0$ denotes (as in section \ref{ClassRLBTorus}) the standard connected component of the fixed point locus $\Pic^0(C)^{\hat\iota}$. 

Concluding, we obtain a diagram
\begin{equation}\label{dia} \begin{array}{ccc}
\pi_0(\Pic^{g-1}(C)^{\hat\iota})&\substack{ \simeq\\ \map\\\pi_0(\varphi)}&\pi_0(\Pic^{u_C}(\Pic^0(C))^{\hat{\hat\iota}})\vspace{2mm} \\
\ \ \ \simeq\   \downarrow {w_C^{g-1}}& &\ \ \ \simeq \ \downarrow w_{\Pic^0(C)}^{u_C}  \vspace{1mm} \\ 
 H^1(C^\iota,\Z_2)_{g-1}&\stackrel{\Phi}{\dasharrow}& W(u_C)
 \end{array}
 \end{equation} 
with  a bijective upper horizontal arrow  and bijective vertical arrows.
\begin{pr} The induced bijection $\Phi: H^1(C^\iota,\Z_2)_{g-1}\to W(u_C)$   is given by the following rule: 

For every $w\in  H^1(C^\iota,\Z_2)_{g-1}$, the element $\Phi(w)\in W(u_C)$ is the unique extension of $f_{u_C}$ satisfying  the equalities
$$\Phi(w)([C_i]^\vee)=\langle w,[C_i]_2\rangle +1 \ ,
$$
where $C_1,\dots C_r$ are the connected components of $C^\iota$.
\end{pr}
\begin{proof}  Let $w\in H^1(C^\iota,\Z_2)_{g-1}$ and let $\Gamma:=(w_C^{g-1})^{-1}(w)\in \pi_0(\Pic^{g-1}(C)^{\hat\iota})$ be the corresponding connected   component of $\Pic^{g-1}(C)^{\hat\iota}$. We know that any connected  component of $\Pic^{g-1}(C)^{\hat\iota}$ contains $2^{g}$ Real theta characteristics (see \cite{GH} p. 169), in particular we can find a Real theta-characteristic $\kappa\in\Gamma$. Using the notations of section \ref{ThetaLB} we can write $w=w(\kappa)$.

Note  now that $\varphi(\kappa)=[{\cal L}_\kappa]$, where ${\cal L}_\kappa:={\cal O}_{\Pic^0(C)}(\Theta_\kappa)$ is the holomorphic line bundle associated with the divisor $\Theta_\kappa$ (see section \ref{ThetaLB}).

Therefore $\pi_0(\varphi)((w_C^{g-1})^{-1}(w))$ is the connected component of $[{\cal L}_\kappa]$ in the fixed point locus $\Pic^{u_C}(\Pic^0(C))^{\hat{\hat\iota}}$, and $\Phi(w)$ is the element of $W(u_C)$ defined by the Stiefel-Whitney class of $\resto{{\cal L}_\kappa^{\tilde\hat{\iota}}}{T_0}$, where $\tilde{\hat\iota}$ is the standard $\hat\iota$-Real structure of ${\cal L}_\kappa$ (see section \ref{ThetaLB}), and where $T_0$ is  the standard connected component of the fixed point locus $\Pic^0(C)^{\hat\iota}$. It suffices to apply Theorem \ref{LK}.
\end{proof}

\begin{co} \label{NiceCo} Let $p_0\in C^\iota$.  Then $\xi:=[{\cal O}_C((g-1)p_0)]\in\Pic^{g-1}(C)^{\hat\iota}$, and  
$$cw_0([{\cal L}_{\xi},\tilde{\hat\iota}])\in \mathrm{Alt}^2(H^1(C,\Z),\Z)^{(\iota^*)^*}\times_{\Hom((\id-\iota^*)H^1(C,\Z),\Z_2)} \Hom(H^1(C,\Z)^{-\iota^*},\Z_2)$$
 is $(u_C,w_{p_0})$, where $w_{p_0}\in W(u_C)$ is the unique extension of $f_{u_C}$ satisfying the equalities:
$$w_{p_0}([C_i]^\vee)=\left\{\begin{array}{ccc} 1&\rm if& p_0\not\in C_i\\
g\hbox{ (mod 2)}&\rm if& p_0\in C_i
\end{array}\right.  
$$
\end{co}

Note that for $\xi:=[{\cal O}_C((g-1)p_0)]\in\Pic^{g-1}(C)^{\hat\iota}$ the Real line bundle $({\cal L}_{\xi},\tilde{\hat\iota})$ is just the Real line bundle ${\cal O}_{\Pic^0(C)}(\Theta-[{\cal O}_C((g-1)p_0)])$ considered in section \ref{famDirac}. According to Proposition \ref{OrSymm} the Stiefel-Whitney of the associated  real line bundle on $\Pic^0(C)^{\hat\iota}$ controls the orientability of the components of $S^d(C)^\iota$ (for $d>2(g-1)$). Therefore Corollary  \ref{NiceCo} together with the difference formula given by Proposition \ref{difference} solves completely the orientability problem formulated in section \ref{RGLSM}.
\\

We conclude with our final result which solves completely the problems formulated in the introduction and in section \ref{famDirac} about Real determinant line bundles of families of Dolbeault operators:
\begin{thry} Let $(C,\iota)$ be a Klein  surface with $C^\iota\ne\emptyset$, $L$   a differentiable line bundle of degree $d$ on $C$, and $[\kappa]\in\Pic^{g-1}(C)^{\hat\iota}$ a  Real theta characteristic. Fix a point $p_0\in C^\iota$,  denote by $\delta^L_{p_0}$, $\bar\partial_{\kappa,p_0}$ the corresponding families of Dolbeault operators parameterized by $\Pic^d(C)$ and $\Pic^0(C)$ respectively, and by  $\det\mathrm{ind}\ \delta_{p_0}^{L}$,  $\det\mathrm{ind}\ \bar\partial_{\kappa,p_0}$  the corresponding determinant line bundles endowed with the $\hat\iota$-Real structures given by Remarks \ref{RealStr1}, \ref{RealStr2}.
\begin{enumerate}
 
\item The element  
$$cw_0([\det\mathrm{ind}\ \bar\partial_{\kappa,p_0}])\in $$
$$  \mathrm{Alt}^2(H^1(C,\Z),\Z)^{(\iota^*)^*}\times_{\Hom((\id-\iota^*)H^1(C,\Z),\Z_2)} \Hom(H^1(C,\Z)^{-\iota^*},\Z_2)$$
  is $(u_C,w_\kappa)$, where $w_\kappa:H^1(C,\Z)^{-\iota^*}\to\Z_2$ is the element of $W(u_C)$ defined as   the unique extension of $f_{u_C}:(\id-\iota^*)H^1(C,\Z)\to\Z_2$ satisfying the equalities
$$
w_\kappa([C_i]^\vee)=\langle w(\kappa),[C_i]\rangle+1\ \hbox{ (mod 2)} \ .$$
\item Let $\tau_{{\cal O}_C(dp_0)}:(\Pic^0(C),\hat\iota)\to(\Pic^d(C),\hat\iota)$ be the isomorphism of Real complex manifolds defined by tensorizing with ${\cal O}_C(dp_0)$. Then the element   
$$cw_0([\left\{\tau_{{\cal O}_C(dp_0)}\right\}^*(\det\mathrm{ind}\ \delta_{p_0}^{L})])\in$$
$$  \mathrm{Alt}^2(H^1(C,\Z),\Z)^{(\iota^*)^*}\times_{\Hom((\id-\iota^*)H^1(C,\Z),\Z_2)} \Hom(H^1(C,\Z)^{-\iota^*},\Z_2)$$
  is $(u_C,w_{p_0})$, where $w_{p_0}\in W(u_C)$ is the unique extension of $f_{u_C}$ satisfying the equalities:
$$w_{p_0}([C_i]^\vee)=\left\{\begin{array}{ccc} 1&\rm if& p_0\not\in C_i\\
g\hbox{ (mod 2)}&\rm if& p_0\in C_i
\end{array}\right.  
$$
\end{enumerate}
\end{thry}
\begin{proof}  
1. This follows   Remark \ref{RealIso} and   Theorem \ref{LK}.
\\ \\
2. This follows   Remark \ref{RealStr2} and Corollary \ref{NiceCo}.
\end{proof}
{\ }\\
{\bf Example:}  Let $(C,\iota)$ be a Real curve of genus 1 with $a(C,\iota)=0$ and $r=2$. Let $C_1$, $C_2$ be the two connected components of $C^\iota$, and let $\Gamma$ be an $\iota$-invariant embedded circle such that $\resto{\iota}{\Gamma}$ is orientation reversing, $[C_1]\cdot[\Gamma]=1$ and $H_1(C,\Z)=\langle [C_1],[\Gamma]\rangle$. The induced involutions on $H_1(C,\Z)$, $H^1(C,\Z)$ are given  by
$$\iota_*([C_1])=[C_1]\ ,\ \iota_*([\Gamma])=-[\Gamma]\ ,\ \iota^*([C_1]^\vee)=-[C_1]^\vee\ ,\ \iota^*([\Gamma]^\vee)=[\Gamma]^\vee.
$$

Identifying $\Pic^0(X)$ with  $H^1(C,{\cal O}_C)/H^1(C,\Z)\simeq [H^1(C,\Z)\otimes_\Z\R]/H^1(C,\Z)$ we see that the involution $\hat\iota$ defined in the introduction is the $\R$-linear extension of  $-\iota^*:  H^1(C,\Z)\to  H^1(C,\Z)$. Therefore 
$$\Pic^0(C)^{\hat\iota}=T_0\cup T_1\hbox{ where } T_0:=\R[C_1]^\vee/\Z[C_1]^\vee ,\ T_1:= \frac{1}{2}[\Gamma]^\vee+ T_0  \ .
$$

Using the notations and  ideas of section \ref{RGLSM} we compute the Stiefel-Whitney class $w_1((\det\mathrm{ind}\ \delta_{p_0}^L)^{\tilde{\hat\iota}_d})$ using the identification $\left\{\tau_{{\cal O}_C(dp_0)}\right\}^*:\Pic^d(C)\to\Pic^0(C)$, Remark \ref{RealIso}  and    Corollary \ref{NiceCo}. Denoting by $T_0'$, $T_1'$ the images of $T_0$, $T_1$ in $\Pic^d(C)^{\hat\iota}$ we obtain
$$w_1\left(\resto{(\det\mathrm{ind}\ \delta_{p_0}^L)^{\tilde{\hat\iota}_d}}{T_0'} \right)([C_1]^\vee)=1\ ,
$$
and using the difference formula Proposition  \ref{difference} we get
$$w_1\left(\resto{(\det\mathrm{ind}\ \delta_{p_0}^L)^{\tilde{\hat\iota}_d}}{T_1'} \right)([C_1]^\vee)=1+u_C([\Gamma]^\vee,[C_1]^\vee])=0\ .
$$
By Proposition \ref{OrSymm} proved in section \ref{RGLSM} we conclude

\begin{pr}\label{Beispiel} Let $(C,\iota)$ be a Real curve of genus 1 with $a(C,\iota)=0$ and $r=2$, and let $d\geq 2$. Then $S^d(C)^\iota$ has two connected components $S^d(C)_{T_0'}^\iota$, $S^d(C)_{T_1'}^\iota$. When $d$ is odd, both components are non-orientable. If $d$ is even, $S^d(C)_{T_0'}^\iota$ is non-orientable and  $S^d(C)_{T_1'}^\iota$ is orientable.
\end{pr}

For instance, when $d=2$ the first component $S^d(C)_{T_0'}^\iota$ is obtained from   $C/\langle\iota\rangle\simeq [0,1]\times S^1$ by gluing two Möbius bands along the two boundary components, so it is non-orientable. The second component $S^d(C)_{T_1'}^\iota$ is just $C_1\times C_2$, so it is orientable.

\section{Appendix}

\centerline{\bf Identities for the Stiefel-Whitney  numbers of}
\vspace{1mm}
\centerline{ \bf  Real vector bundles} 
 \vspace{6mm} 

Let $(X,\iota)$ be a closed, connected, oriented  differentiable $n$-manifold  endowed with an involution $\iota$, $X^\iota$ the fixed point locus of $\iota$, and let $(E,\tilde\iota)$ be a $\iota$-Real  vector bundle of rank $r$ on $X$.   We denote by $_\R E $ the underlying real bundle of $E$, and by $E^{\tilde\iota}$ the fixed point locus of $\tilde\iota$, regarded as a real bundle of rang $r$ on $X^\iota$.

\begin{re}\label{standard} Let $X\supset Y\textmap{\theta} X^\iota$ be the projection (with $\Z_2$-invariant  fibers) of a sufficiently small tubular neighborhood  $Y$ of $X^\iota$. Then the restriction $\resto{(E,\tilde\iota)}{Y}$ can be identified with $\theta^*(E^{\tilde\iota})\otimes\C$  endowed with the involution defined by conjugation.
\end{re} 
\begin{proof} The map of $\Z_2$-spaces $\theta:Y\to X^\iota$ is a $\Z_2$-homotopy equivalence, so the result follows from \cite{A1}.
\end{proof}

Suppose $n$ is even, and  let $\cg=\Pi_{i=1}^r c_i^{k_i}$ be a Chern monomial  of degree  $2\sum  i k_i=n$, and let $\wg=\Pi_{i=1}^r w_{2i}^{k_i}$ be the corresponding Stiefel-Whitney  monomial.
We will show that the  Stiefel-Whitney number 
$$\langle \wg(_\R E),[X]_2\rangle\equiv \langle \cg(E),[X]\rangle  \hbox{ (mod 2)}\ $$
can be computed by a polynomial expression in the Stiefel-Whitney classes of the real bundle $E^{\tilde \iota}$ over $X^\iota$ and of the normal bundle of $X^\iota$ in $X$.

Consider the $(n+1)$-dimensional fiber  product $Q:=X\times_{\Z_2} S^1$, where $\Z_2$ acts on $X$ via $\iota$ and on $S^1$ via the antipodal involution. This fiber product can be regarded as the total space of the locally trivial bundle $p:Q\to \P^1_\R$  with fiber $X$ associated with the principal $\Z_2$-bundle $S^1\to\P^1_\R$. By definition  $Q$ can be also regarded as the base of a principal $\Z_2$-bundle $q:X\times S^1\to X$. We define a $\Z_2=\{\pm1\}$-action on $Q$   by
$$(-1)\cdot [x,x_0,x_1]):=[x,x_0,-x_1]=[\iota( x),- x_0,x_1]\ .
$$
This action  lifts the $\Z_2$-action on $\P^1_\R$ induced by the standard orientation reversing reflection with fixed points $0=[1,0]$, $\infty=[0,1]\in \P^1_\R$. $Q$ decomposes as a union of two $\Z_2$-invariant open set $Q_0\supset p^{-1}(0)$, $Q_\infty\supset p^{-1}(\infty)$ both diffeomorphic to $X\times\R$ via the maps 
$$(x,x_1)\stackrel{j_0}{\mapsto} [x,1,x_1]\ ,\ (x,x_0)\stackrel{j_\infty}\mapsto [x,x_0,1]\ .$$
 The induced actions on $X\times\R$ are
$$ (x,x_1)\mapsto (x, -x_1)\ ,\  (x,x_0)\mapsto (\iota( x),-x_0)\ .
$$

The fixed point locus  $Q^{\Z_2}$ of the  $\Z_2$-action on $Q$ decomposes as
$$Q^{\Z_2}=Q_0^{\Z_2}\cup Q_\infty^{\Z_2}\ ,
$$
where $Q^{\Z_2}_0=p^{-1}(0)$ is naturally isomorphic with $X$, and $Q_\infty^{\Z_2}\subset p^{-1}(\infty)$ can be identified with $X^\iota$. The normal bundle of $Q^{\Z_2}_0\simeq X$ in $Q$ is $X\times \R(1)$  (trivial line bundle with the $\Z_2$-action induced by $-\id$), whereas the normal bundle of $Q_\infty^{\Z_2}\simeq X^\iota$ in $Q$  is  $N_{X^\iota/X}(1)\times\R(1)$.

Let $U:=U_0\coprod U_\infty$ be an open $\Z_2$-equivariant tubular neighborhood of $Q^{\Z_2}$. The quotient
$$W:=\qmod{[Q\setminus U]}{\Z_2}
$$
is a compact manifold with boundary
$$\partial W=X\coprod \P_\R(N_{X^\iota/X}\oplus\underline\R) 
$$
consisting of $X$ and  of the real projectivization of the normal bundle of $X^\iota$ in $Q$. The bundle $p_1^*(_\R E)$ on $X\times S^1$ comes with an obvious $\Z_2$-action and descends to $Q$ via $q$, because $q$ is the quotient with respect to a free $\Z_2$-action. The total space  of the descended bundle is the fiber product $F=E\times_{\Z_2} S^1$, where $\Z_2$ acts on $E$ via $\tilde\iota$.  The pull-backs  $j_0^*(F)$, $j_\infty^*(F)$ can both be  identified with $p_X^*(E)=E\times\R$. The $\Z_2$-action  on $Q$ can be lifted to $F$ using the formula
$$(-1)\cdot [e,x_0,x_1]):=[e,x_0,- x_1]=[\tilde\iota(e),- x_0,x_1]\ ,
$$
and the induced actions on   $j_0^*(F)=E\times\R$, $j_\infty^*(F)=E\times\R$ are
$$(e,x_1)\mapsto (e,- x_1)\ ,\ (e,x_0)\mapsto (\tilde\iota(e),- x_0)\ .
$$

Using this lift we see that   the restriction $\resto{F}{Q\setminus  U}$ descends to $W$. Using the identifications  $Q_0^{\Z_2}=X$, $Q_\infty^{\Z_2}=\P_\R(N_{X^\iota/X}\oplus\underline{\R})$,  denoting by $\chi$ the tautological real line bundle on this projective bundle, and by $\pi$ the projection $\P_\R(N_{X^\iota/X}\oplus\underline{\R})\to X^\iota$, we see that the restrictions of the obtained bundle $\bar F$  to the two parts of $\partial W$  are
$$\resto{\bar F}{Q_0^{\Z_2}}=E\ ,\ \resto{\bar F}{Q_\infty^{\Z_2}}=\pi^*(E^{\tilde\iota})\oplus [\pi^*(E^{\tilde\iota})\otimes\chi]\ .
$$
For the second formula we used Remark \ref{standard} and the obvious $\R$-isomorphism of   $\Z_2$-bundles $E^{\tilde\iota}\otimes\C\simeq E^{\tilde\iota}\oplus E^{\tilde\iota}(1)$.
\\
\\ 
Applying the  Whitney formula one can decompose:
$$\wg\left(\pi^*(E^{\tilde\iota})\oplus [\pi^*(E^{\tilde\iota})\otimes\chi]\right)=a_0w_1(\chi)^{n} +\sum_{0< ij\leq n } a_{ij} w_i^j(\pi^*(E^{\tilde\iota}))w_1(\chi)^{n-ij}\ .
$$
Suppose that   $X^\iota$ has constant codimension $k$ at any point, and  denote by $\eta$ the real $(k+1)$-bundle $N_{X^\iota/X}\oplus\underline{\R}$ on $X^\iota$. Note that $w_i(\eta)=w_i(N_{X^\iota/X})$, in particular $w_{k+1}(\eta)=0$. We have 
$$\pi_*(w_1(\chi)^{k+l})=s_l(\eta), 
$$
where $s(\eta)=\sum_{l=0}^\infty s_l(\eta)=w(\eta)^{-1}$.  Hence we obtain
$$\pi_*\left\{\wg\left(\pi^*(E^{\tilde\iota})\oplus [\pi^*(E^{\tilde\iota})\otimes\chi]\right)\right\}=a_0 s_{n-k}(\eta)+\sum_{l=1}^{n-k} \sum_{ij=l} a_{ij} w_i^j(E^{\tilde\iota}) s_{n-k-l}(\eta)\ .
$$

Regarding $W$ as a  homology equivalence between its boundary parts $X$ and $\P_\R(N_{X^\iota/X}\oplus\underline{\R})$, and using the identity
$$\langle \sigma, [\P_\R(N_{X^\iota/X}\oplus\underline{\R})]_2\rangle=\langle \pi_*(\sigma), [X^\iota]_2\rangle
$$
we get the following {\it localization formula}:
\begin{thry} Let $(X,\iota)$ be a closed, connected, oriented  differentiable $n$-manifold  endowed with an involution $\iota$, $X^\iota$ the fixed point locus of $\iota$, and let $(E,\tilde\iota)$ be a $\iota$-Real   vector bundle  of rank $r$ on $X$. Then for every Chern monomial $\cg$ of degree $n$ with corresponding Stiefel-Whitney monomial $\wg$ we have
$$\left\langle \wg(_\R E),[X]_2\right\rangle=\left\langle a_0 s_{n-k}(N_{X^\iota/X})+\sum_{l=1}^{n-k} \sum_{ij=l} a_{ij} w_i^j(E^{\tilde\iota}) s_{n-k-l}(N_{X^\iota/X}),[X^\iota]_2\right\rangle\ .
$$
\end{thry}
\vspace{1mm}{\ }\\
{\bf Example:} $n=2$, $k=1$, $\wg=w_2$. In this case one has: 
$$w_2\left(\pi^*(E^{\tilde\iota})\oplus [\pi^*(E^{\tilde\iota})\otimes\chi]\right)=w_1(\pi^*(E^{\tilde\iota}))(w_1(\pi^*(E^{\tilde\iota}))+r w_1(\chi))+ w_2(\pi^*(E^{\tilde\iota})\otimes\chi)+$$
$$+  w_2(\pi^*(E^{\tilde\iota}))=  w_1(\pi^*(E^{\tilde\iota}))w_1(\chi)+\frac{r(r-1)}{2}w_1^2(\chi)+ w_1(\pi^*(E^{\tilde\iota}))^2=$$
$$=w_1(\pi^*(E^{\tilde\iota}))w_1(\chi)+\frac{r(r-1)}{2}w_1^2(\chi)\ .
$$

Here we used have the general formula:
$$w_2(F\otimes\chi)=w_2(F)+ (r-1) w_1(F)w_1(\chi)+\frac{r(r-1)}{2}w_1^2(\chi)\ .
$$
Therefore

\begin{co}  Let $(E,\tilde\iota)$ be a $\iota$-Real   vector bundle  of rank $r$ on  a closed  Real 2-manifold
$(X,\iota)$ with $X^\iota$ of codimension 1. Then
$$\langle w_2(_\R E),[X]_2\rangle =\langle w_1(E^{\tilde\iota})+\frac{r(r-1)}{2} w_1(N_{X^\iota/X}),  [X^\iota]_2 \rangle  \ .
$$ 
\end{co}

\end{document}